\documentclass[10pt,a4paper]{article}
\usepackage{fullpage,mathrsfs,graphicx,framed,color,amssymb,amsmath,amsthm}
\usepackage[boxed, noline, ruled, linesnumbered]{algorithm2e}
\usepackage{epsfig, graphicx}
\usepackage{latexsym,amsfonts,amsbsy,amssymb}
\usepackage{amsmath,amsthm}
\usepackage{enumerate}
\usepackage{bm}
\usepackage{slashbox} 
\makeatletter 

\@addtoreset{equation}{section}
\makeatother 
\allowdisplaybreaks 
\newtheorem{theorem}{Theorem}[section]
\newtheorem{lemma}{Lemma}[section]

\newtheorem{remark}{Remark}[section]
\newtheorem{asume}{Assumption}[section]

\begin{document}
\title{Optimal error estimates of sequential finite element method for nonlinear thermo-poroelasticity problems}
\author{Fan Chen\footnote{Department of Mathematics and Statistics, Zaozhuang University, Zaozhuang 277160, China.
({\tt chenfan00@126.com})},\ \ \
Changqing Lv\footnote{Department of Mathematics and Statistics, Zaozhuang University, Zaozhuang 277160, China. ({\tt cqiqc1999@126.com})} \  \  and\  \ Chenguang Zhou\footnote{Department of Mathematics, Beijing University of Technology, Beijing 100124, China. ({\tt zhoucg@bjut.edu.cn})} \footnote{Corresponding author: Chenguang Zhou} }
\date{}
\maketitle
\begin{abstract}
This study introduces and analyzes a three-step sequential decoupling algorithm designed to address nonlinear, fully coupled quasi-static thermo-poroelasticity systems incorporating convective transport. The finite element method is employed for spatial discretization and the backward Euler method for temporal discretization. The proposed sequential method has a higher computing efficiency than the fully implicit nonlinear numerical scheme, since it does not require any internal iterations. The well-posedness of the numerical solution is discussed by introducing a cut-off operator and the stability analysis of the algorithm is performed. Rigorous analysis yields optimal convergence order estimates for both spatial and temporal discretizations. In order to confirm the theoretical results and the effectiveness of the suggested approach, numerical experiments are finally carried out.

\end{abstract}

\section{Introduction}\label{intro}
Poroelasticity theory, originating from foundational contributions by Terzaghi \cite{Terzaghi1943Theoretical} and Biot \cite{12}, provides a rigorous mathematical framework for modeling the coupled interactions between fluid flow and solid deformation within porous media. This theory has been subsequently extended to the non-isothermal case \cite{2010Macroscopic,MR3920899}, giving rise to the thermo-poroelasticity model that incorporates thermal effects into the poroelasticity system. Owing to its robust physical basis and broad applicability, the thermo-poroelasticity framework is extensively utilized across diverse scientific and engineering domains. For instance, it plays a critical role in simulating soil consolidation dynamics in geotechnical engineering \cite{poromechanics}, analyzing energy transfer mechanisms in geothermal systems, modeling thermal ablation processes in tumor treatment within biomechanics \cite{nzl}, and addressing environmental challenges such as carbon dioxide sequestration and thermal remediation strategies \cite{2021A,2017Multi}.

 Let $\Omega\subset\mathbb R^d$ $(d=2,3)$ be a convex polygonal or polyhedral domain with a Lipschitz boundary $\partial\Omega$, and let $(0,t_f]$ be a time interval with $t_f>0$. This research examines the following nonlinear thermo-poroelasticity model: Find $({\mathbf u}, p, T)$ satisfying
\begin{align}
 -\nabla\cdot\bm\sigma({\mathbf u})+\alpha\nabla p+\beta\nabla T
&=\bm{f}, \ \ \textrm {in} \  \Omega\times(0,t_f], \label{e1}\\
\frac {\partial }{\partial t}(c_0p-b_0T+\alpha\nabla\cdot{\mathbf u})-\nabla\cdot({\mathbf K}\nabla p)&=g, \ \  \textrm {in}  \ \Omega\times(0,t_f],\label{e2}\\
\frac {\partial }{\partial t}(a_0T-b_0p+\beta\nabla\cdot{\mathbf u})-{\mathbf K}\nabla p\cdot\nabla T-\nabla\cdot(\bm\Theta\nabla T)&=z, \ \  \textrm {in}  \ \Omega\times(0,t_f].\label{e3}
\end{align}
Here, $\bm f$ denotes the body force, while $g$ and $z$ represent the mass source and the heat source respectively. The displacement field is expressed as ${\mathbf u}(t):\Omega\to\mathbb R^d$, with pressure $p(t):\Omega\to\mathbb R$ and temperature $T(t):\Omega\to\mathbb R$ representing the corresponding scalar fields. The effective stress tensor is given by the expression $\bm\sigma({\mathbf u})=2\mu\bm\epsilon({\mathbf u})+\lambda\nabla\cdot{\mathbf u\mathbf I}$ with the strain tensor $\bm\epsilon({\mathbf u})=\frac{\nabla{\mathbf u}+\nabla{\mathbf u}^{\mathsf T}}{2}$, the Lam\'{e} constants $\mu$ and $\lambda$, and the identity tensor ${\mathbf I}$. The coefficients $\alpha$ and $\beta$ denote the Biot-Willis constant and the thermal stress, respectively. The parameter $a_0$ denotes the effective volumetric heat capacity, $b_0$ the thermal dilation coefficient, and $c_0 \geq 0$ the constrained specific storage coefficient. The permeability tensor divided by fluid viscosity is denoted by ${\mathbf K}=(K_{ij})_{i,j=1}^d$ and the thermal conductivity tensor is represented by $\bm\Theta=(\Theta_{ij})_{i,j=1}^d$.

The system is subject to initial conditions at $t=0$,
\begin{align}
{\mathbf u}(\cdot,0)&={\mathbf u}^0, \ \ \textrm {in} \ \Omega,\label{e5}\\
p(\cdot,0)&=p^0, \ \ \textrm {in} \ \Omega,\label{e6}\\
T(\cdot,0)&=T^0, \ \ \textrm {in} \ \Omega,\label{e7}
\end{align}
and boundary conditions on $\partial\Omega\times(0,t_f]$,
\begin{align}
{\mathbf u}=\mathbf 0,\ \ p=0,\ \ T&=0.\label{e4}
\end{align}

A variety of numerical methods have been developed for the thermo-poroelasticity model, including finite element methods \cite{MR5009266,MR5024217,MR5030055,MR4432106}, discontinuous Galerkin methods \cite{MR4579739,chen}, enriched Galerkin methods \cite{MR4755187}, etc.. From the perspective of computational efficiency, the innate complexity of coupled thermo-poroelastic multi-physics systems requires the development of numerical methods that are both accurate and computationally efficient. Large-scale algebraic systems must be built and solved for fully implicit techniques \cite{chen}, which couple and solve all variables at each time step. Even though these techniques have good stability properties, their high computational expense is a major barrier to large-scale engineering applications.

Numerical techniques that decouple the fluid and elasticity subproblems have gained popularity as a natural alternative for solving large-scale coupled systems. For the thermo-poroelasticity model, two primary decoupling approaches exist: iterative method and sequential method. Iterative methods \cite{MR4685928,MR3614285,2019Monolithic,MR3253752} solve subproblems at each time step, relying on approximate solutions from previous iterations until a predefined convergence tolerance is achieved. While this approach effectively reduces the computational complexity per iteration step, it introduces additional costs due to the iterative process, particularly in strongly nonlinear regimes, where convergence guarantees become challenging to establish. Sequential methods \cite{MR3800024,MR4101493}, solve the subproblems at each time step without internal iterations. Compared with fully coupled algorithms, this strategy significantly reduces the dimension of the discrete problem and enhances the efficiency of computation. In the recent work, Chaabane and Rivi\`ere \cite{MR3777103} introduce a new sequential finite element method for the isothermal linear Biot system, where the system is decoupled into two subproblems, and pressure and displacement are solved in time step sequence. Through incorporating stabilization terms, a priori error estimates of numerical scheme for displacement and pressure are derived. Nevertheless, their analysis is limited to convergence results that are optimal for piecewise linear elements but suboptimal for high order elements by imposing $\Delta t \ge h^2$ (see Remark 4.3 of \cite{MR3777103} for more details). Moreover, the stability of their scheme remains unaddressed.

This paper presents a sequential decoupling method for solving fully coupled quasi-static thermo-poroelasticity problems that include nonlinear convective transport. At each time step, this approach consists of three steps. First, the fluid equation is solved for pressure, under the assumption that the temperature and displacement from the previous two time layers are available. Second, using the computed pressure and the heat equation, the temperature is updated based on the known displacement values from the two prior time layers. Third, the displacement  is obtained by solving the mechanical equation, utilizing the previously computed pressure and temperature. The thermo-poroelasticity model has thus been completely separated at every time step.  Clearly, the values of pressure, temperature, and displacement from the first two time layers are required for the implementation of the above procedure. Here, we select a linear numerical scheme similar to that of \cite{MR4432106} to obtain the optimal convergence. The finite element method and the backward Euler method are used for discretizing space and time. We examine the algorithm's stability as well as the existence and uniqueness of numerical solutions. Furthermore, we derive error estimates that attain optimal convergence rates in both spatial and temporal discretizations. Unlike the fully implicit method \cite{MR4432106} or the iterative method \cite{cai}, our method only solves once at each time step and there is no internal iterations, which illustrates that the proposed method has higher computational efficiency. Different from the existing partitioned approaches, our method has the following features: (i) we use a linear extrapolation of displacement and temperature from the previous two time layers to decouple the pressure equation; (ii) a single stabilization parameter $\gamma$ controls the error propagation between the subproblems; (iii) the theoretical analysis covers both the stability conclusion and optimal error estimates in space and time, which are not fully addressed in the earlier work \cite{MR3777103}. To the best of our knowledge, this work is the first to establish both stability and optimal error estimates for the sequential finite element method applied to the fully coupled nonlinear thermo-poroelasticity system.

This paper is structured as follows. In Section \ref{CA}, we carry out some preliminaries and present the fully discrete numerical scheme for the thermo-poroelasticity problem. The details of the sequential decoupling algorithm are provided in Section \ref{SD}, where we also perform an algorithmic well-posedness study and derive the optimal convergence order estimates. Finally, in Section \ref{NE}, we report our numerical simulations to validate our theoretical findings and expectations.

\section{Coupled algorithm}\label{CA}
In this section, we firstly present the variational formulation of the model problem \eqref{e1}-\eqref{e7}. Then, a fully discrete finite element scheme is constructed and its stability analysis is conducted.

\subsection{Preliminaries}
In this article, we define the following spaces. The Sobolev space $W^{k,l}(\Omega)$ (cf. \cite{MR450957}) in $L^l(\Omega)$ is employed. Specifically, we denote $H^{k}(\Omega):=W^{k,2}(\Omega)$ for $k\geq 0$. Let $(\cdot, \cdot)_k$ and $\|\cdot\|_k$ denote the inner-product and norm on $H^k(\Omega)$, respectively. For $k=0$, $H^0(\Omega)$  is equivalent to $L^2(\Omega)$, which represents the space of square-integrable functions, the inner product and norm are denoted without the subscript $k$. Define $H_0^1(\Omega)$ as the subspace of $ H^1(\Omega)$ containing functions with vanishing trace on $\partial\Omega$. We denote by $||\cdot||_{0,\infty}$ the norm on $L^{\infty}(\Omega)$.
For functions dependent on both time and space, we define
$ L^2(0,t_f;H^k(\Omega))=\{v:\int_0^{t_f}||v||_k^2dt<\infty\}$
with
$||v||_{L^2(0,t_f;H^k(\Omega))}^2=\int_0^{t_f}||v||_{k}^2dt.$

The following assumptions, taken from \cite{2018Well}, are introduced before the numerical scheme for problem \eqref{e1}-\eqref{e7} is presented.
\begin{asume}\label{js}
({1}) Both the thermal conductivity $\bm\Theta$ and the permeability ${\mathbf K}$ are constant in time, symmetric and positive definite. Furthermore, there exist $k_m>0$, $k_M>0$ and $\theta_m>0$, $\theta_M>0$, such that
\begin{align*}
k_m|\zeta|^2\leq \zeta^T{\mathbf K}(\mathbf x)\zeta, \ \ |{\mathbf K}(\mathbf x)| \leq k_M|\zeta|,\quad \forall \zeta\in \mathbb R^d, \quad\mathbf x \in \Omega,\\
\theta_m|\zeta|^2\leq \zeta^T\bm\Theta(\mathbf x)\zeta,\ \ |\bm\Theta(\mathbf x)|\leq \theta_M|\zeta|,\quad \forall \zeta\in \mathbb R^d,\quad \mathbf x \in \Omega.
\end{align*}
({2}) Strictly positive constants are defined as $a_0$, $b_0$, $c_0$, $\alpha$, $\beta$, $\mu$, and $\lambda$.\\
({3}) The parameters satisfy $c_0-2b_0>0$ and $a_0-2b_0\ge0$.\\
({4}) The right-hand side terms  fulfil  $z,g\in L^2(0,t_f;L^2(\Omega))$ and ${\bm f}\in H^1(0,t_f;L^2(\Omega))$.\\
({5}) The initial values  fulfil $p^0,T^0\in H_0^1(\Omega)$ and ${\mathbf u}^0\in(H_0^1(\Omega))^d$.\\
({6}) The solution functions $\nabla p$, $p$, $T$, $\mathbf u$ satisfy, for $l>d$ and $k_i\geq1$ $(i=1,2,3)$,
\begin{align*}
&||\nabla p||_{L^{\infty}(0,t_f;W^{1,l}(\Omega))}+||p||_{L^{\infty}(0,t_f;W^{2,l}(\Omega))}+||p||_{L^{\infty}(0,t_f;H^{k_2+1}(\Omega))}\\
+&||T||_{L^{\infty}(0,t_f;H^{k_3+1}(\Omega))}+
||\mathbf u||_{L^{\infty}(0,t_f;(H^{k_1+1}(\Omega))^d)}+||p_t||_{L^{\infty}(0,t_f;W^{1,l}(\Omega))}\\
+&||T_t||_{L^{\infty}(0,t_f;W^{1,l}(\Omega))}+
||\mathbf u_t||_{L^{\infty}(0,t_f;(W^{2,l}(\Omega))^d)}+||p_{tt}||_{L^{2}(0,t_f;L^{l}(\Omega))}\\
+&||T_{tt}||_{L^{2}(0,t_f;L^{l}(\Omega))}+
||\mathbf u_{tt}||_{L^{2}(0,t_f;(W^{1,l}(\Omega))^d)}+
||\mathbf u_{ttt}||_{L^{2}(0,t_f;(L^{2}(\Omega))^d)}\le C.
\end{align*}
\end{asume}
Additionally, the sign $C$, with or without subscripts, represents a generic positive constant that may vary depending on where it appears in the document.

For simplicity, we present the spaces
$${\mathbf V}=(H_0^1(\Omega))^d, \quad Q=H_0^1(\Omega),\quad W=H_0^1(\Omega),$$
and the bilinear forms
\begin{align*}
&a({\mathbf u},{\mathbf v})=2\mu(\bm\epsilon({\mathbf u}),\bm\epsilon({\mathbf v}))+\lambda(\nabla\cdot\mathbf u,\nabla\cdot\mathbf v),\\
&c_p(p,q)=(\mathbf K\nabla p,\nabla q),\\
&c_T(T,s)=(\bm\Theta\nabla T,\nabla s),\\
&b(p,{\mathbf v})=(\nabla p,{\mathbf v}).
\end{align*}
Then, the corresponding variational formulation for \eqref{e1}-\eqref{e3} reads as: for any $t\in(0,t_f]$, find $(\mathbf u,p,T)\in \mathbf V\times Q\times W$ such that
 \begin{align}
a({\mathbf u},{\mathbf v})+\alpha b(p,{\mathbf v})+\beta b(T,{\mathbf v})&=({\bm f},{\mathbf v}),\quad \forall {\mathbf v}\in \mathbf V, \label{b1}\\
c_0(p_t,q)-b_0(T_t,q)-\alpha b(q,{\mathbf u}_t)+c_p(p,q)&=(g,q),\quad \forall q\in Q,\label{b2}\\
a_0(T_t,s)-b_0(p_t,s)-\beta b(s,{\mathbf u}_t)-({\mathbf K}\nabla p\cdot\nabla T,s)+c_T(T,s)&=(z,s),\quad \forall s\in W,\label{b3}
\end{align}
with the initial conditions \eqref{e5}-\eqref{e7}.

Following the similar idea of \cite{MR4931528}, we supply the energy analysis below.
\begin{lemma}
For $t\in(0,t_f]$, each weak solution $(\mathbf u,p,T)$ of problems \eqref{b1}-\eqref{b3} satisfies the following energy dissipation law:
\begin{align*}
\frac{d}{d t} E(t)+k_m||\nabla p||^2+(\theta_m-C_pk_M||\nabla p||_{0,\infty})||\nabla T||^2\le (g,p)+(z,T)-(\partial_t\bm f,\mathbf u),
 \end{align*}
 where
 \begin{align*}
E(t):=&\frac{1}{2}(2\mu||\bm\epsilon({\mathbf u}(t))||^2+\lambda||\nabla\cdot\mathbf u(t)||^2+(c_0-b_0)||p(t)||^2\\
&+(a_0-b_0)||T(t)||^2+b_0||p(t)-T(t)||^2)-(\bm f(t),\mathbf u(t)).
 \end{align*}
\end{lemma}
\begin{proof}
Choosing ${\mathbf v}={\mathbf u}_t$ in \eqref{b1}, $q=p$ in \eqref{b2} and $s=T$ in \eqref{b3}, and adding them together, we can get the energy equation directly.
\end{proof}

\subsection{Fully discrete numerical scheme}\label{VF}
Consider a regular triangulation $\mathcal T_h$ of $\Omega$. For each triangle $K\in\mathcal T_h$, we denote its diameter by $h_K$, and define $h$ as the maximum diameter among all triangles in the triangulation, i.e., $h=\max\limits_{K\in \mathcal T_h} h_K.$  In order to provide a finite element method, let $P_{k_i}(K)$ $(i=1,2,3)$ represent the function space consisting of all polynomials defined on element $K$ with degree not exceeding $k_i$. Then, we define the discrete approximate spaces
\begin{align*}
{\mathbf V}_h&=\{{\mathbf v}\in {\mathbf V}:{\mathbf v}|_{K}\in (P_{k_1}(K))^d,\ \forall K\in \mathcal T_h\},\\
Q_h&=\{q\in Q:q|_{K}\in P_{k_2}(K),\ \forall K\in\mathcal T_h\},\\
S_h&=\{w\in W:w|_{K}\in P_{k_3}(K),\ \forall K\in\mathcal T_h\}.
\end{align*}

As stated in \cite{2019Monolithic,MR2039576,MR2116915}, we introduce the following cut-off operator $\mathcal M$,
\begin{equation}\label{m}
\mathcal M({\mathbf z})(x):=\left\{ \begin{array}{l}
{\mathbf z}(x) , \hspace{2cm} |{\mathbf z}(x)|\le M,\\
M{\mathbf z}(x)/|{\mathbf z}(x)|, \hspace{0.6cm} |{\mathbf z}(x)|> M,
\end{array} \right.
\end{equation}
in which $M$ is a positive constant. This operator is mainly intended to facilitate convergence analysis for the numerical scheme. It is straightforward to verify that $\mathcal M({\mathbf K}\nabla p^n)(x)={\mathbf K}\nabla p^n(x)$ holds true, provided that the exact Darcy flux is bounded, i.e., ${\mathbf K}\nabla p^n \in (L^{\infty}(\Omega))^d$, and $ M $ is selected to be adequately large. Therefore, it is not necessary to specify the exact value of $M$.

\begin{lemma}\label{lem0}\cite{MR2116915} The cut-off operator $\mathcal M$ is uniformly Lipschitz continuous, that is,
\begin{equation*}
||\mathcal M(\mathbf z_1)-\mathcal M(\mathbf z_2)||_{0,\infty}\le||\mathbf z_1-\mathbf z_2||_{0,\infty}.
\end{equation*}
Then, it follows that
\begin{align*}
||\mathcal M({\mathbf K}\nabla p^n)-\mathcal M({\mathbf K}\nabla p_h^n)||_{0,\infty}\le&||{\mathbf K}\nabla p^n-{\mathbf K}\nabla p_h^n||_{0,\infty},\\
||\mathcal M({\mathbf K}\nabla p_h^n)||_{0,\infty}\le& M.
\end{align*}
\end{lemma}

In order to supply the fully discrete numerical scheme, we define the time step size $\tau=\frac{t_f}{N}$ for some positive integer $N$ and $t_n=n\tau$ for $n=0,1,\cdots, N$. The quantities ${\mathbf u}_h^n$, $p_h^n$ and $T_h^n$ serve as discrete approximations to the exact solutions ${\mathbf u}(t_n)$, $p(t_n)$ and $T(t_n)$ at time $t_n$, respectively. For brevity, we introduce the simplified notations ${\bm f}^n$, $g^n$ and $z^n$ to represent ${\bm f}(t_n)$, $g(t_n)$ and $z(t_n)$ independently. Then, the fully discrete numerical scheme is expressed as follows using the backward Euler method to approximate the time derivative in \eqref{b2} and \eqref{b3}: For $n=0,\cdots, N-1$, find $({\mathbf u}_h^{n+1}, p_h^{n+1},T_h^{n+1})\in {\mathbf V}_h\times Q_h\times S_h$,  such that for any $({\mathbf v}_h,q_h,s_h)\in {\mathbf V}_h\times Q_h\times S_h$,
 \begin{align}
a({\mathbf u}_h^{n+1},{\mathbf v_h} )+\alpha b(p_h^{n+1},{\mathbf v}_h)+\beta b(T_h^{n+1},{\mathbf v}_h)&=({\bm f}^{n+1},{\mathbf v_h}),\label{b7}\\
c_0(\partial_\tau p_h^{n+1},q_h)-b_0(\partial_\tau T_h^{n+1},q_h)-\alpha b(q_h, \partial_\tau{\mathbf u}_h^{n+1})+c_p(p_h^{n+1},q_h)&=(g^{n+1},q_h),\label{b8}\\
a_0(\partial_\tau T_h^{n+1},s_h)-b_0(\partial_\tau p_h^{n+1},s_h)-\beta b(s_h,\partial_\tau{\mathbf u}_h^{n+1})+c_T(T_h^{n+1},s_h)&\notag\\
-(\mathcal M(\mathbf K\nabla p_h^{n})\cdot\nabla T_h^{n+1},s_h)&=(z^{n+1},s_h),\label{b9}
\end{align}
with the initial conditions
\begin{align}
{\mathbf u}_h^0 = R_u{\mathbf u}^0,\ \ p_h^0 = R_pp^0,\ \ T_h^0 = R_TT^0,  \quad \textrm{in} \ \Omega.\label{b10}
\end{align}
Here, $\partial_\tau \bm u_h^{n+1}=\frac{\bm u_h^{n+1}-\bm u_h^{n}}{\tau}$, $\partial_\tau p_h^{n+1}=\frac{p_h^{n+1}-p_h^{n}}{\tau}$ and $\partial_\tau T_h^{n+1}=\frac{T_h^{n+1}-T_h^{n}}{\tau}$. The notations ${R}_u$, $R_p$ and $R_T$ represent projection operators with the following definitions: For the displacement, we define the projection $R_u{\mathbf u}\in {\mathbf V}_h$ \cite{wheeler} to satisfy
\begin{equation}
a(\mathbf u-R_u\mathbf u, {\mathbf v}_h)=0,\quad \forall {\mathbf v}_h\in {\mathbf V}_h.\label{ea1}
\end{equation}
For the pressure, define the projection $R_pp\in Q_h$ \cite{wheeler} satisfying
\begin{equation}
c_p(p-R_pp, q_h)=0,\quad \forall q_h\in Q_h.\label{ea2}
\end{equation}
For the temperature, define the projection $R_TT\in S_h$ \cite{MR4432106} satisfying
\begin{equation}
c_T(T-R_TT, s_h)-({\mathbf K}\nabla p\cdot\nabla(T-R_TT),s_h)+\delta(T-R_TT,s_h)=0,\quad \forall s_h\in S_h,\label{ea3}
\end{equation}
where the constant $\delta$ is chosen to be large enough to insure the coercivity of the bilinear form. It is easy to see that from the coercivity and boundedness of $a(\cdot,\cdot)$, $c_p(\cdot,\cdot)$ and $c_T(\cdot,\cdot)$, the Lax-Milgram theorem \cite{MR2519594} implies that the operators $R_u$, $R_p$ and $R_T$ are well-defined.

In the above scheme \eqref{b9}, we utilize $(\mathcal M(\mathbf K\nabla p_h^{n})\cdot\nabla T_h^{n+1},s_h)$ instead of the approximating term $(\mathbf K\nabla p_h^{n}\cdot\nabla T_h^{n+1},s_h).$ As the authors state in \cite{2019Monolithic}, the introduction of cut-off operator has little or no practical implications, but is necessary in order to facilitate the convergence analysis (the main function is to provide the result of Lemma \ref{lem0} in this paper). In the reference \cite{MR4432106}, the authors have investigated the standard Galerkin method for the thermo-poroelasticity model \eqref{e1}-\eqref{e3}. They establish the semi-discrete and fully discrete finite element schemes with the approximating term $(\mathbf K\nabla p_h^{n}\cdot\nabla T_h^{n+1},s_h)$ and obtain the stability of the method. The error estimates for the displacement, pressure and temperature are derived. The coupled scheme \eqref{b7}-\eqref{b9} serves as the starting point for our sequential decoupling strategy; its well-posedness and stability analysis are established in the coming Theorem \ref{thm2.1} following the framework of \cite{MR4432106}. This analysis is included here for completeness and to provide a baseline for comparison with the sequential scheme, whose stability conclusion and distinct error structure will be analyzed in Section \ref{SD}.

Before the stability of \eqref{b7}-\eqref{b9}, we introduce the following two lemmas.

\begin{lemma}\label{lem2.3}\cite{MR2519594} For $\forall q\in Q_h \ (or\ S_h)$, we have
\begin{equation*}
||q||^2\leq C_p||\nabla q||^2.
\end{equation*}
\end{lemma}

\begin{lemma}\label{lem2.4}\cite{MR2519594} For $\forall {\mathbf v}\in {\mathbf V}_h$, we have
\begin{align*}
||{\mathbf v}||_1^2\leq C_k||\bm\epsilon({\mathbf v})||^2.
\end{align*}
\end{lemma}

The stability study of the suggested fully discrete numerical scheme \eqref{b7}-\eqref{b9} is then shown. The definition of the total energy $E^{n+1}$ at time $t_{n+1}$ is given as
\begin{align}
E^{n+1}:=&\frac{1}{2}(2\mu||\bm\epsilon({\mathbf u_h^{n+1}})||^2+\lambda||\nabla\cdot\mathbf u_h^{n+1}||^2+(c_0-b_0)||p_h^{n+1}||^2\notag\\
&+(a_0-b_0)||T_h^{n+1}||^2+b_0||p_h^{n+1}-T_h^{n+1}||^2).\label{deE}
\end{align}

\begin{theorem}\label{thm2.1}
Under the definition of \eqref{deE}, the fully discrete numerical scheme \eqref{b7}-\eqref{b9} is stable for $\theta_m=2MC_p$.
\end{theorem}
\begin{proof}
Selecting $\mathbf v_h=\mathbf u_h^{n+1}-\mathbf u_h^n$ in \eqref{b7}, $q_h=\tau p_h^{n+1}$ in \eqref{b8}, $s_h=\tau T_h^{n+1}$ in \eqref{b9}, and then summing the equations up, we have
\begin{align}
&a(\mathbf u_h^{n+1},\mathbf u_h^{n+1}-\mathbf u_h^{n})+(c_0-b_0)(p_h^{n+1}-p_h^{n},p_h^{n+1})+(a_0-b_0)(T_h^{n+1}-T_h^{n},T_h^{n+1})\notag\\
&+\tau c_p(p_h^{n+1},p_h^{n+1})+\tau c_T(T_h^{n+1},T_h^{n+1})+b_0(p_h^{n+1}-T_h^{n+1}-p_h^{n}+T_h^{n},p_h^{n+1}-T_h^{n+1})\notag\\
=&(\bm f^{n+1},\mathbf u_h^{n+1}-\mathbf u_h^{n})+\tau(g^{n+1},p_h^{n+1})+\tau(z^{n+1},T_h^{n+1})+\tau(\mathcal M(\mathbf K\nabla p_h^{n})\cdot\nabla T_h^{n+1},T_h^{n+1}).\label{qsp1}
\end{align}
It should be noted that
\begin{align*}
a(\mathbf u_h^{n+1},\mathbf u_h^{n+1}-\mathbf u_h^{n})
=&\mu(||\bm\epsilon({\mathbf u_h^{n+1}})||^2-||\bm\epsilon({\mathbf u_h^{n}})||^2+||\bm\epsilon({\mathbf u_h^{n+1}})-\bm\epsilon({\mathbf u_h^{n}})||^2)\notag\\
&+\frac{\lambda}{2}(||\nabla\cdot\mathbf u_h^{n+1}||^2-||\nabla\cdot\mathbf u_h^{n}||^2+||\nabla\cdot\mathbf u_h^{n+1}-\nabla\cdot\mathbf u_h^{n}||^2),\\
(p_h^{n+1}-p_h^{n},p_h^{n+1})=&\frac{1}{2}(||p_h^{n+1}||^2-||p_h^{n}||^2+||p_h^{n+1}-p_h^{n}||^2),\\
(T_h^{n+1}-T_h^{n},T_h^{n+1})=&\frac{1}{2}(||T_h^{n+1}||^2-||T_h^{n}||^2+||T_h^{n+1}-T_h^{n}||^2),\\
(p_h^{n+1}-T_h^{n+1}-p_h^{n}+T_h^{n},p_h^{n+1}-T_h^{n+1})
=&\frac{1}{2}(||p_h^{n+1}-T_h^{n+1}||^2-||p_h^{n}-T_h^{n}||^2\\
&+||p_h^{n+1}-T_h^{n+1}-p_h^{n}+T_h^{n}||^2).
\end{align*}
Taking above equations in \eqref{qsp1}, we obtain
\begin{align}
&\mu(||\bm\epsilon({\mathbf u_h^{n+1}})||^2-||\bm\epsilon({\mathbf u_h^{n}})||^2+||\bm\epsilon({\mathbf u_h^{n+1}})-\bm\epsilon({\mathbf u_h^{n}})||^2)\notag\\
&+\frac{\lambda}{2}(||\nabla\cdot\mathbf u_h^{n+1}||^2-||\nabla\cdot\mathbf u_h^{n}||^2+||\nabla\cdot\mathbf u_h^{n+1}-\nabla\cdot\mathbf u_h^{n}||^2)\notag\\
&+\frac{c_0-b_0}{2}(||p_h^{n+1}||^2-||p_h^{n}||^2+||p_h^{n+1}-p_h^{n}||^2)\notag\\
&+\frac{a_0-b_0}{2}(||T_h^{n+1}||^2-||T_h^{n}||^2+||T_h^{n+1}-T_h^{n}||^2)\notag\\
&+\tau {\mathbf K}||\nabla p_h^{n+1}||^2+\tau {\bm\Theta}||\nabla T_h^{n+1}||^2\notag\\
&+\frac{b_0}{2}(||p_h^{n+1}-T_h^{n+1}||^2-||p_h^{n}-T_h^{n}||^2+||p_h^{n+1}-p_h^{n}-(T_h^{n+1}-T_h^{n})||^2) \notag\\
=&(\bm f^{n+1},\mathbf u_h^{n+1}-\mathbf u_h^{n})+\tau(g^{n+1},p_h^{n+1})+\tau(z^{n+1},T_h^{n+1})+\tau(\mathcal M(\mathbf K\nabla p_h^{n})\cdot\nabla T_h^{n+1},T_h^{n+1}).\label{qsp11}
\end{align}
Then, we have
\begin{align}
& E^{n+1}+\tau k_m||\nabla p_h^{n+1}||^2+\tau \theta_m||\nabla T_h^{n+1}||^2+\mu||\bm\epsilon({\mathbf u_h^{n+1}})-\bm\epsilon({\mathbf u_h^{n}})||^2+\frac{\lambda}{2}||\nabla\cdot\mathbf u_h^{n+1}-\nabla\cdot\mathbf u_h^{n}||^2
\notag\\
&+\frac{c_0-b_0}{2}||p_h^{n+1}-p_h^{n}||^2+\frac{a_0-b_0}{2}||T_h^{n+1}-T_h^{n}||^2+\frac{b_0}{2}||p_h^{n+1}-T_h^{n+1}-p_h^{n}+T_h^{n}||^2 \notag\\
=&E^{n}+(\bm f^{n+1},\mathbf u_h^{n+1}-\mathbf u_h^{n})+\tau(g^{n+1},p_h^{n+1})+\tau(z^{n+1},T_h^{n+1})+\tau(\mathcal M(\mathbf K\nabla p_h^{n})\cdot\nabla T_h^{n+1},T_h^{n+1}).\label{qsp2}
\end{align}
The Poincar\'e inequality and Korn's inequality allow us to bound the right-hand side terms by
\begin{align}
&\tau(g^{n+1},p_h^{n+1})+\tau(z^{n+1},T_h^{n+1})\le\frac{k_m\tau}{2}||\nabla p_h^{n+1}||^2+\frac{\tau C_p^2}{2k_m}||g^{n+1}||^2+\frac{\theta_m\tau}{2}||\nabla T_h^{n+1}||^2+\frac{\tau C_p^2}{2\theta_m}||z^{n+1}||^2,\label{qsp6}\\
&\tau(\mathcal M(\mathbf K\nabla p_h^{n})\cdot\nabla T_h^{n+1},T_h^{n+1})\le \tau MC_p||\nabla T_h^{n+1}||^2.\label{qsp7}
\end{align}
Taking \eqref{qsp6}-\eqref{qsp7} in \eqref{qsp2}, we obtain
\begin{align}
& E^{n+1}+\frac{k_m\tau}{2}||\nabla p_h^{n+1}||^2+(\frac{\theta_m}{2}-MC_p)\tau||\nabla T_h^{n+1}||^2\notag\\
&+\mu||\bm\epsilon({\mathbf u_h^{n+1}})-\bm\epsilon({\mathbf u_h^{n}})||^2+\frac{\lambda}{2}||\nabla\cdot\mathbf u_h^{n+1}-\nabla\cdot\mathbf u_h^{n}||^2+\frac{c_0-b_0}{2}||p_h^{n+1}-p_h^{n}||^2
\notag\\
&+\frac{a_0-b_0}{2}||T_h^{n+1}-T_h^{n}||^2+\frac{b_0}{2}||p_h^{n+1}-T_h^{n+1}-p_h^{n}+T_h^{n}||^2 \notag\\
\le&E^{n}+(\bm f^{n+1},\mathbf u_h^{n+1}-\mathbf u_h^{n})+\frac{\tau C_p^2}{2k_m}||g^{n+1}||^2+\frac{\tau C_p^2}{2\theta_m}||z^{n+1}||^2.\label{qsp21}
\end{align}
Summing up \eqref{qsp21} from $n = 1$ to $N-1$, we have
\begin{align}
& E^{N}+\frac{k_m\tau}{2}\sum\limits_{n=1}^{N-1}||\nabla p_h^{n+1}||^2+(\frac{\theta_m}{2}-MC_p)\tau\sum\limits_{n=1}^{N-1}||\nabla T_h^{n+1}||^2\notag\\
&+\mu\sum\limits_{n=1}^{N-1}||\bm\epsilon({\mathbf u_h^{n+1}})-\bm\epsilon({\mathbf u_h^{n}})||^2+\frac{\lambda}{2}\sum\limits_{n=1}^{N-1}||\nabla\cdot\mathbf u_h^{n+1}-\nabla\cdot\mathbf u_h^{n}||^2+\frac{c_0-b_0}{2}\sum\limits_{n=1}^{N-1}||p_h^{n+1}-p_h^{n}||^2
\notag\\
&+\frac{a_0-b_0}{2}\sum\limits_{n=1}^{N-1}||T_h^{n+1}-T_h^{n}||^2+\frac{b_0}{2}\sum\limits_{n=1}^{N-1}||p_h^{n+1}-T_h^{n+1}-p_h^{n}+T_h^{n}||^2 \notag\\
\le&E^{1}+\sum\limits_{n=1}^{N-1}(\bm f^{n+1},\mathbf u_h^{n+1}-\mathbf u_h^{n})+\frac{\tau C_p^2}{2k_m}\sum\limits_{n=1}^{N-1}||g^{n+1}||^2+\frac{\tau C_p^2}{2\theta_m}\sum\limits_{n=1}^{N-1}||z^{n+1}||^2.\label{xj1}
\end{align}
Applying
\begin{equation}
\sum\limits_{n=1}^{N-1}A^{n+1}(B^{n+1}-B^{n})=A^NB^N-A^1B^1-\sum\limits_{n=1}^{N-1}(A^{n+1}-A^{n})B^{n},\label{f}
\end{equation}
we obtain
\begin{align}
&\sum\limits_{n=1}^{N-1}(\bm f^{n+1},\mathbf u_h^{n+1}-\mathbf u_h^{n})=\tau \sum\limits_{n=1}^{N-1}(\bm f^{n+1},\partial_\tau\mathbf u_h^{n+1})\notag\\
=&(\bm f^{N},\mathbf u_h^{N})-(\bm f^{1},\mathbf u_h^{1})-\tau \sum\limits_{n=1}^{N-1}(\partial_\tau\bm f^{n+1},\mathbf u_h^{n})\notag\\
\le& \frac{C_p^2C_k^2}{2\mu}||\bm f^{N}||^2+\frac{\mu}{2}||\bm\epsilon({\mathbf u_h^{N}})||^2+\frac{C_p^2C_k^2}{2\mu}||\bm f^{1}||^2+\frac{\mu}{2}||\bm\epsilon({\mathbf u_h^{1}})||^2\notag\\
&+ \frac{\tau C_p^2C_k^2}{2\mu} \sum\limits_{n=1}^{N-1}||\partial_\tau\bm f^{n+1}||^2+\frac{\tau\mu }{2} \sum\limits_{n=1}^{N-1}||\bm\epsilon({\mathbf u_h^{n}})||^2.
\end{align}
So we have
\begin{align}
& E^{N}+\frac{k_m\tau}{2}\sum\limits_{n=1}^{N-1}||\nabla p_h^{n+1}||^2+(\frac{\theta_m}{2}-MC_p)\tau\sum\limits_{n=1}^{N-1}||\nabla T_h^{n+1}||^2\notag\\
&+\mu\sum\limits_{n=1}^{N-1}||\bm\epsilon({\mathbf u_h^{n+1}})-\bm\epsilon({\mathbf u_h^{n}})||^2+\frac{\lambda}{2}\sum\limits_{n=1}^{N-1}||\nabla\cdot\mathbf u_h^{n+1}-\nabla\cdot\mathbf u_h^{n}||^2+\frac{c_0-b_0}{2}\sum\limits_{n=1}^{N-1}||p_h^{n+1}-p_h^{n}||^2
\notag\\
&+\frac{a_0-b_0}{2}\sum\limits_{n=1}^{N-1}||T_h^{n+1}-T_h^{n}||^2+\frac{b_0}{2}\sum\limits_{n=1}^{N-1}||p_h^{n+1}-T_h^{n+1}-p_h^{n}+T_h^{n}||^2 \notag\\
\le&E^{1}+\frac{\mu}{2}||\bm\epsilon(\mathbf u_h^{N})||^2+\frac{\mu}{2}||\bm\epsilon(\mathbf u_h^{1})||^2+\frac{C_p^2C_k^2}{2\mu}(||\bm f^{1}||^2+||\bm f^{N}||^2+\tau\sum\limits_{n=1}^{N-1}||\partial_\tau\bm f^{n+1}||^2)\notag\\
&+\frac{\tau C_p^2}{2k_m}\sum\limits_{n=1}^{N-1}||g^{n+1}||^2+\frac{\tau C_p^2}{2\theta_m}\sum\limits_{n=1}^{N-1}||z^{n+1}||^2+\frac{\mu\tau}{2}\sum\limits_{n=1}^{N-1}||\bm\epsilon(\mathbf u_h^{n})||^2.\label{xj0}
\end{align}
Combining with Gronwall inequality, the proof is completed.
\end{proof}

\section{Sequentially decoupled algorithm}\label{SD}
In this section, we introduce and examine a sequential decoupling algorithm. A stabilizer is employed to address the stability issue brought on by the coupling term in this fully decoupled algorithm. The existence, uniqueness and stability of the numerical solution of the method are discussed, and the optimal convergence order estimate in both space and time is derived.

\subsection{Sequential scheme}
For $n\geq1$, the proposed three-step sequential scheme corresponding to \eqref{b1}-\eqref{b3} is constructed as follows.\\
{\bf Step1} Given $p_h^n\in Q_h$, $T_h^n$, $T_h^{n-1}\in S_h$ and
${\mathbf u}_h^n$, ${\mathbf u}_h^{n-1}\in {\mathbf V}_h$, find $p_h^{n+1}\in Q_h$ such that
\begin{align}\label{s1}
&c_0(\frac{p_h^{n+1}-p_h^{n}}{\tau},q_h)-b_0(\frac{T_h^{n}-T_h^{n-1}}{\tau},q_h)-\alpha b(q_h,\frac{{\mathbf u}_h^{n}-{\mathbf u}_h^{n-1}}{\tau})\notag\\
&+c_p(p_h^{n+1},q_h)=(g^{n+1},q_h),\quad \forall q_h\in Q_h.
\end{align}
{\bf Step2} Given $T_h^n\in S_h$, $p_h^n$, $p_h^{n+1}\in Q_h$ and
${\mathbf u}_h^n$, ${\mathbf u}_h^{n-1}\in {\mathbf V}_h$, find $T_h^{n+1}\in S_h$ such that
\begin{align}\label{s2}
&a_0(\frac{T_h^{n+1}-T_h^{n}}{\tau},s_h)-b_0(\frac{p_h^{n+1}-p_h^{n}}{\tau},s_h)-\beta b(s_h,\frac{{\mathbf u}_h^{n}-{\mathbf u}_h^{n-1}}{\tau})\notag\\
&+c_T(T_h^{n+1},s_h)-(\mathcal M(\mathbf K\nabla p_h^{n})\cdot\nabla T_h^{n+1},s_h)=(z^{n+1},s_h),\quad \forall s_h\in S_h.
\end{align}
{\bf Step3} Given $p_h^{n+1}\in Q_h,$ $T_h^{n+1}\in S_h$ and
${\mathbf u}_h^n$, ${\mathbf u}_h^{n-1}\in {\mathbf V}_h$, find ${\mathbf u}_h^{n+1}\in {\mathbf V}_h$ such that
\begin{align}\label{s3}
&a({\mathbf u}_h^{n+1},{\mathbf v_h} )+\alpha b(p_h^{n+1},{\mathbf v}_h)+\beta b(T_h^{n+1},{\mathbf v}_h)\notag\\
&+\gamma(\frac{{\mathbf u}_h^{n+1}-{\mathbf u}_h^{n}}{\tau},{\mathbf v_h})-\gamma(\frac{{\mathbf u}_h^{n}-{\mathbf u}_h^{n-1}}{\tau},{\mathbf v_h})=({\bm f}^{n+1},{\mathbf v_h}),\quad \forall {\mathbf v}_h\in {\mathbf V}_h.
\end{align}
Here, $\gamma$ is the stabilization parameter \cite{MR3742883} independent of the mesh and time step sizes, and only depends on the material parameters.

To initialize the algorithm, the solution $p_h^1$, $T_h^1$ and ${\mathbf u}_h^1$ needs to be estimated beforehand. The initial condition $(p_h^0,T_h^0,{\mathbf u}_h^0)$ of the algorithm \eqref{s1}-\eqref{s3} follows \eqref{b10}. From \eqref{s1}-\eqref{s3}, it can be seen that directly solving for $(p_h^1, T_h^1, {\mathbf u}_h^1)$ is unfeasible, since the formulations \eqref{s1}-\eqref{s3} require the values of $T_h^n$ and $T_h^{n-1}$, as well as $\mathbf u_h^n$ and $\mathbf u_h^{n-1}$. Hence, we adopt the linear numerical scheme \eqref{b7}-\eqref{b9} to compute $(p_h^1,T_h^1,{\mathbf u}_h^1)$ by using $(p_h^0,T_h^0,{\mathbf u}_h^0)$ and the optimal convergence order can be obtained.

\subsection{Well-posedness and stability}\label{wp}
In this subsection, we establish the well-posedness and stability of the sequentially decoupled algorithm \eqref{s1}-\eqref{s3}. First, the existence and uniqueness of the numerical solution of the sequential decoupling method is demonstrated.

\begin{theorem}\label{eu}
Assuming that Assumption \ref{js} hold, the algorithm \eqref{s1}-\eqref{s3} has a unique solution $({\mathbf u}_h^n,p_h^n,T_h^n)\in {\mathbf V}_h\times Q_h\times S_h$ when $\tau \le\frac{2\theta_ma_0}{M^2}$.
\end{theorem}
\begin{proof}
Since the formulations \eqref{s1}-\eqref{s3} are linear and defined in finite-dimensional spaces, the uniqueness of the solution directly implies its existence. Let $(\mathbf{u}_1, p_1, T_1)$ and $(\mathbf{u}_2, p_2, T_2)$ be the two solutions of the discrete system \eqref{s1}-\eqref{s3} with the same given data from the previous time levels. Define the differences
$$\bar{\mathbf{u}}_h = \mathbf{u}_1 - \mathbf{u}_2,\quad \bar{p}_h = p_1 - p_2,\quad \bar{T}_h = T_1 - T_2.$$
Substituting these into \eqref{s1}-\eqref{s3} and subtracting yields the homogeneous system, for $n\geq 2$,
\begin{align}
c_0(\frac{\bar{p}_h}{\tau},q_h)+c_p(\bar{p}_h,q_h)=&0,\label{wy1}\\
a_0(\frac{\bar{T}_h}{\tau},s_h)-(\mathcal M({\mathbf K}\nabla p_h^{n})\cdot\nabla \bar{T}_h,s_h)+c_T(\bar{T}_h,s_h)=&0,\label{wy2}\\
a(\bar{\mathbf u}_h,{\mathbf v_h} )+\gamma(\frac{\bar{\mathbf u}_h}{\tau},{\mathbf v_h})=&0.\label{wy3}
\end{align}
Choosing $q_h=\bar{p}_h$ in \eqref{wy1}, we get
\begin{equation*}
\frac{c_0}{\tau}||\bar{p}_h||^2+k_M||\nabla\bar{p}_h||^2\leq0.
\end{equation*}
Therefore, we obtain $\bar{p}_h=0.$

Next, applying Lemma \ref{lem0} and $s_h=\bar{T}_h$ to \eqref{wy2}, we provide
\begin{align*}
a_0(\frac{\bar{T}_h}{\tau},\bar{T}_h)+c_T(\bar{T}_h,\bar{T}_h)-(\mathcal M({\mathbf K}\nabla p_h^{n})\cdot\nabla \bar T_h, \bar T_h)
\geq &\frac{a_0}{\tau}||\bar{T}_h||^2+\theta_m||\nabla\bar{T}_h||^2-M||\nabla \bar T_h||||\bar T_h||\\
\geq& (\frac{a_0}{\tau}-\frac{M^2}{2\theta_m})||\bar{T}_h||^2+\frac{\theta_m}{2}||\nabla\bar{T}_h||^2.
\end{align*}
When $\frac{a_0}{\tau}-\frac{M^2}{2\theta_m}\ge0$, we render
\begin{equation*}
(\frac{a_0}{\tau}-\frac{M^2}{2\theta_m})||\bar{T}_h||^2+\frac{\theta_m}{2}||\nabla\bar{T}_h||^2\leq0.
\end{equation*}
Therefore, we obtain $\bar{T}_h=0.$

Finally, by virtue of ${\mathbf v_h}=\bar{\mathbf u}_h$ in \eqref{wy3}, there holds
\begin{equation*}
2\mu||\bm\epsilon(\bar{\mathbf u}_h)||^2+\lambda||\nabla\cdot\bar{\mathbf u}_h||^2+\frac{\gamma}{\tau}||\bar{\mathbf u}_h||^2\leq0.
\end{equation*}
That is, $\bar{\mathbf u}_h=0$. The uniqueness of the discrete solutions $p_h^1$, $T_h^1$ and ${\mathbf u}_h^1$ follows by an analogous argument to that employed in the preceding analysis.
\end{proof}

Next, we conduct the stability analysis of the proposed algorithm \eqref{s1}-\eqref{s3}.

\begin{theorem}\label{thm0}
Let $\mathbf u_h^{n}$, $p_h^{n}$ and $T_h^{n}$ be the solution of the algorithm \eqref{s1}-\eqref{s3}. Then, for $\theta_m=2M C_p$ and $\gamma=2(\frac{\alpha^2}{k_m}+\frac{\beta^2}{\theta_m})$, the system \eqref{s1}-\eqref{s3} is stable.
\end{theorem}
\begin{proof}
Taking $q_h=\tau p_h^{n+1}$ in \eqref{s1}, $s_h=\tau T_h^{n+1}$ in \eqref{s2}, $\mathbf v_h=\mathbf u_h^{n+1}-\mathbf u_h^n$ in \eqref{s3}, and then summing the equations up, we have
\begin{align}
&a(\mathbf u_h^{n+1},\mathbf u_h^{n+1}-\mathbf u_h^{n})+(c_0-b_0)(p_h^{n+1}-p_h^{n},p_h^{n+1})+(a_0-b_0)(T_h^{n+1}-T_h^{n},T_h^{n+1})\notag\\
&+\gamma\tau(\partial_\tau\mathbf u_h^{n+1}-\partial_\tau\mathbf u_h^{n},\partial_\tau\mathbf u_h^{n+1})+\tau c_p(p_h^{n+1},p_h^{n+1})+\tau c_T(T_h^{n+1},T_h^{n+1})\notag\\
&+b_0(T_h^{n+1}-T_h^n,p_h^{n+1})+b_0(p_h^{n+1}-T_h^{n+1}-p_h^n+T_h^n,p_h^{n+1}-T_h^{n+1})\notag\\
=&(\bm f^{n+1},\mathbf u_h^{n+1}-\mathbf u_h^{n})+\tau(g^{n+1},p_h^{n+1})+\tau(z^{n+1},T_h^{n+1})+\tau(\mathcal M(\mathbf K\nabla p_h^{n})\cdot\nabla T_h^{n+1},T_h^{n+1})\notag\\
 &-\alpha b(p_h^{n+1},\mathbf u_h^{n+1}-2\mathbf u_h^{n}+\mathbf u_h^{n-1})-\beta b(T_h^{n+1},\mathbf u_h^{n+1}-2\mathbf u_h^{n}+\mathbf u_h^{n-1})\notag\\
  &+b_0(T_h^{n}-T_h^{n-1},p_h^{n})+b_0(T_h^{n}-T_h^{n-1},p_h^{n+1}-p_h^{n}).\label{ssp1}
\end{align}

Using the identities already stated above in Theorem \ref{thm2.1}, we obtain
\begin{align}
&\mu(||\bm\epsilon({\mathbf u_h^{n+1}})||^2-||\bm\epsilon({\mathbf u_h^{n}})||^2+||\bm\epsilon({\mathbf u_h^{n+1}}-{\mathbf u_h^{n}})||^2)+\frac{\lambda}{2}(||\nabla\cdot\mathbf u_h^{n+1}||^2-||\nabla\cdot\mathbf u_h^{n}||^2+||\nabla\cdot(\mathbf u_h^{n+1}-\mathbf u_h^{n})||^2)\notag\\
&+\frac{c_0-b_0}{2}(||p_h^{n+1}||^2-||p_h^{n}||^2+||p_h^{n+1}-p_h^{n}||^2)+\frac{a_0-b_0}{2}(||T_h^{n+1}||^2-||T_h^{n}||^2+||T_h^{n+1}-T_h^{n}||^2)\notag\\
&+\frac{\gamma\tau}{2}(||\partial_\tau\mathbf u_h^{n+1}||^2-||\partial_\tau\mathbf u_h^{n}||^2+||\partial_\tau\mathbf u_h^{n+1}-\partial_\tau\mathbf u_h^{n}||^2)+\tau c_p(p_h^{n+1},p_h^{n+1})+\tau c_T(T_h^{n+1},T_h^{n+1})\notag\\
&+\frac{b_0}{2}(||p_h^{n+1}-T_h^{n+1}||^2-||p_h^{n}-T_h^{n}||^2+||p_h^{n+1}-T_h^{n+1}-p_h^{n}+T_h^{n}||^2)+b_0(T_h^{n+1}-T_h^n,p_h^{n+1}) \notag\\
=&(\bm f^{n+1},\mathbf u_h^{n+1}-\mathbf u_h^{n})+\tau(g^{n+1},p_h^{n+1})+\tau(z^{n+1},T_h^{n+1})+\tau(\mathcal M(\mathbf K\nabla p_h^{n})\cdot\nabla T_h^{n+1},T_h^{n+1})\notag\\
 &-\alpha b(p_h^{n+1},\mathbf u_h^{n+1}-2\mathbf u_h^{n}+\mathbf u_h^{n-1})-\beta b(T_h^{n+1},\mathbf u_h^{n+1}-2\mathbf u_h^{n}+\mathbf u_h^{n-1})\notag\\
  &+b_0(T_h^{n}-T_h^{n-1},p_h^{n})+b_0(T_h^{n}-T_h^{n-1},p_h^{n+1}-p_h^{n}),\label{ssp2}
\end{align}
i.e.,
\begin{align}
&E^{n+1}+\frac{\gamma\tau}{2}||\partial_\tau\mathbf u_h^{n+1}||^2+b_0(T_h^{n+1}-T_h^n,p_h^{n+1})+\mu||\bm\epsilon({\mathbf u_h^{n+1}}-{\mathbf u_h^{n}})||^2+\frac{\lambda}{2}||\nabla\cdot(\mathbf u_h^{n+1}-\mathbf u_h^{n})||^2\notag\\
&+\frac{c_0-b_0}{2}||p_h^{n+1}-p_h^{n}||^2+\frac{a_0-b_0}{2}||T_h^{n+1}-T_h^{n}||^2+\tau c_p(p_h^{n+1},p_h^{n+1})+\tau c_T(T_h^{n+1},T_h^{n+1})\notag\\
&+\frac{\gamma\tau}{2}||\partial_\tau\mathbf u_h^{n+1}-\partial_\tau\mathbf u_h^{n}||^2
+\frac{b_0}{2}||p_h^{n+1}-T_h^{n+1}-p_h^{n}+T_h^{n}||^2  \notag\\
=&E^{n}+\frac{\gamma\tau}{2}||\partial_\tau\mathbf u_h^{n}||^2+b_0(T_h^{n}-T_h^{n-1},p_h^{n})+(\bm f^{n+1},\mathbf u_h^{n+1}-\mathbf u_h^{n})+\tau(g^{n+1},p_h^{n+1})\notag\\
&+\tau(z^{n+1},T_h^{n+1})-\alpha b(p_h^{n+1},\mathbf u_h^{n+1}-2\mathbf u_h^{n}+\mathbf u_h^{n-1})-\beta b(T_h^{n+1},\mathbf u_h^{n+1}-2\mathbf u_h^{n}+\mathbf u_h^{n-1})\notag\\
&+\tau(\mathcal M(\mathbf K\nabla p_h^{n})\cdot\nabla T_h^{n+1},T_h^{n+1})+b_0(T_h^{n}-T_h^{n-1},p_h^{n+1}-p_h^{n}).\label{ssp3}
\end{align}
Due to the Young's inequality, we render
\begin{align}
&|\alpha b(p_h^{n+1},\mathbf u_h^{n+1}-2\mathbf u_h^{n}+\mathbf u_h^{n-1})|\le\frac{k_m\tau}{4}||\nabla p_h^{n+1}||^2+\frac{\tau\alpha^2}{k_m}||\partial_\tau\mathbf u_h^{n+1}-\partial_\tau\mathbf u_h^{n}||^2,\label{ssp4}\\
&|\beta b(T_h^{n+1},\mathbf u_h^{n+1}-2\mathbf u_h^{n}+\mathbf u_h^{n-1})|\le\frac{\theta_m\tau}{4}||\nabla T_h^{n+1}||^2+\frac{\tau\beta^2}{\theta_m}||\partial_\tau\mathbf u_h^{n+1}-\partial_\tau\mathbf u_h^{n}||^2,\label{ssp41}\\
&b_0(T_h^{n}-T_h^{n-1},p_h^{n+1}-p_h^{n})\le \frac{b_0}{2}||T_h^{n}-T_h^{n-1}||^2+\frac{b_0}{2}||p_h^{n+1}-p_h^{n}||^2.\label{ssp8}
\end{align}
Taking \eqref{ssp4}-\eqref{ssp8} and \eqref{qsp6}-\eqref{qsp7} in \eqref{ssp3}, we supply
\begin{align}
&E^{n+1}+\frac{\gamma\tau}{2}||\partial_\tau\mathbf u_h^{n+1}||^2+\frac{b_0}{2}||T_h^{n+1}-T_h^{n}||^2+b_0(T_h^{n+1}-T_h^n,p_h^{n+1})\notag\\
&+\mu||\bm\epsilon({\mathbf u_h^{n+1}}-{\mathbf u_h^{n}})||^2+\frac{\lambda}{2}||\nabla\cdot(\mathbf u_h^{n+1}-\mathbf u_h^{n})||^2+\frac{c_0-2b_0}{2}||p_h^{n+1}-p_h^{n}||^2\notag\\
&+\frac{a_0-2b_0}{2}||T_h^{n+1}-T_h^{n}||^2+\frac{\tau k_m}{2}||\nabla p_h^{n+1}||^2+(\frac{\tau \theta_m}{2}-\tau MC_p)||\nabla T_h^{n+1}||^2\notag\\
&+\frac{b_0}{2}||p_h^{n+1}-T_h^{n+1}-p_h^{n}+T_h^{n}||^2+(\frac{\gamma}{2}-(\frac{\alpha^2}{k_m}+\frac{\beta^2}{\theta_m}))\tau||\partial_\tau\mathbf u_h^{n+1}-\partial_\tau\mathbf u_h^{n}||^2
  \notag\\
\le&E^{n}+\frac{\gamma\tau}{2}||\partial_\tau\mathbf u_h^{n}||^2+\frac{b_0}{2}||T_h^{n}-T_h^{n-1}||^2+b_0(T_h^{n}-T_h^{n-1},p_h^{n})\notag\\
&+(\bm f^{n+1},\mathbf u_h^{n+1}-\mathbf u_h^{n})+\frac{\tau C_p^2}{k_m}||g^{n+1}||^2+\frac{\tau C_p^2}{\theta_m}||z^{n+1}||^2.\label{ssp9}
\end{align}
In order to guarantee that the last term on the left side of \eqref{ssp9} is positive, we should select $\gamma=2(\frac{\alpha^2}{k_m}+\frac{\beta^2}{\theta_m})$. The stated estimate can be proven by summing over $1\le n\le N-1$,
\begin{align}
&E^{N}+\frac{\gamma\tau}{2}||\partial_\tau\mathbf u_h^{N}||^2+\frac{b_0}{2}||T_h^{N}-T_h^{N-1}||^2\notag\\
&+\mu\sum\limits_{n=1}^{N-1}||\bm\epsilon({\mathbf u_h^{n+1}}-{\mathbf u_h^{n}})||^2+\frac{\lambda}{2}\sum\limits_{n=1}^{N-1}||\nabla\cdot(\mathbf u_h^{n+1}-\mathbf u_h^{n})||^2+\frac{c_0-2b_0}{2}\sum\limits_{n=1}^{N-1}||p_h^{n+1}-p_h^{n}||^2\notag\\
&+\frac{a_0-2b_0}{2}\sum\limits_{n=1}^{N-1}||T_h^{n+1}-T_h^{n}||^2+\frac{\tau k_m}{2}\sum\limits_{n=1}^{N-1}||\nabla p_h^{n+1}||^2+(\frac{\tau \theta_m}{2}-\tau MC_p)\sum\limits_{n=1}^{N-1}||\nabla T_h^{n+1}||^2\notag\\
&+\frac{b_0}{2}\sum\limits_{n=1}^{N-1}||p_h^{n+1}-T_h^{n+1}-p_h^{n}+T_h^{n}||^2+(\frac{\gamma}{2}-(\frac{\alpha^2}{k_m}+\frac{\beta^2}{\theta_m}))\tau\sum\limits_{n=1}^{N-1}||\partial_\tau\mathbf u_h^{n+1}-\partial_\tau\mathbf u_h^{n}||^2
  \notag\\
\le&E^{1}+\frac{\gamma\tau}{2}||\partial_\tau\mathbf u_h^{1}||^2+\frac{b_0}{2}||T_h^{1}-T_h^{0}||^2-b_0(T_h^{N}-T_h^{N-1},p_h^{N})+b_0(T_h^{1}-T_h^{0},p_h^{1})\notag\\
&+\sum\limits_{n=1}^{N-1}(\bm f^{n+1},\mathbf u_h^{n+1}-\mathbf u_h^{n})+\frac{\tau C_p^2}{k_m}\sum\limits_{n=1}^{N-1}||g^{n+1}||^2+\frac{\tau C_p^2}{\theta_m}\sum\limits_{n=1}^{N-1}||z^{n+1}||^2.\label{ssp10}
\end{align}

From the Young's inequality,
\begin{align}
&|-b_0(T_h^{N}-T_h^{N-1},p_h^{N})|\le\frac{b_0}{2}||T_h^{N}-T_h^{N-1}||^2+\frac{b_0}{2}||p_h^{N}||^2,\\
&|b_0(T_h^{1}-T_h^{0},p_h^{1})|\le\frac{b_0}{2}||T_h^{1}-T_h^{0}||^2+\frac{b_0}{2}||p_h^{1}||^2,
\end{align}

Applying \eqref{f}, similar to \eqref{xj0}, it follows that
\begin{align}
&E^{N}+\frac{\gamma\tau}{2}||\partial_\tau\mathbf u_h^{N}||^2+\mu\sum\limits_{n=1}^{N-1}||\bm\epsilon({\mathbf u_h^{n+1}}-{\mathbf u_h^{n}})||^2\notag\\
&+\frac{\lambda}{2}\sum\limits_{n=1}^{N-1}||\nabla\cdot(\mathbf u_h^{n+1}-\mathbf u_h^{n})||^2+\frac{c_0-2b_0}{2}\sum\limits_{n=1}^{N-1}||p_h^{n+1}-p_h^{n}||^2\notag\\
&+\frac{a_0-2b_0}{2}\sum\limits_{n=1}^{N-1}||T_h^{n+1}-T_h^{n}||^2+\frac{\tau k_m}{2}\sum\limits_{n=1}^{N-1}||\nabla p_h^{n+1}||^2+(\frac{\tau \theta_m}{2}-\tau MC_p)\sum\limits_{n=1}^{N-1}||\nabla T_h^{n+1}||^2\notag\\
&+\frac{b_0}{2}\sum\limits_{n=1}^{N-1}||p_h^{n+1}-T_h^{n+1}-p_h^{n}+T_h^{n}||^2+(\frac{\gamma}{2}-(\frac{\alpha^2}{k_m}+\frac{\beta^2}{\theta_m}))\tau\sum\limits_{n=1}^{N-1}||\partial_\tau\mathbf u_h^{n+1}-\partial_\tau\mathbf u_h^{n}||^2
  \notag\\
\le&E^{1}+\frac{\gamma\tau}{2}||\partial_\tau\mathbf u_h^{1}||^2+b_0||T_h^{1}-T_h^{0}||^2+\frac{b_0}{2}||p_h^{1}||^2+\frac{b_0}{2}||p_h^{N}||^2\notag\\
&+\frac{C_p^2C_k^2}{2\mu}||\bm f^{N}||^2+\frac{\mu}{2}||\bm\epsilon({\mathbf u_h^{N}})||^2+\frac{C_p^2C_k^2}{2\mu}||\bm f^{1}||^2+\frac{\mu}{2}||\bm\epsilon({\mathbf u_h^{1}})||^2+\frac{\tau \mu}{2} \sum\limits_{n=1}^{N-1}||\bm\epsilon({\mathbf u_h^{n}})||^2\notag\\
&+ \frac{\tau C_p^2C_k^2}{2\mu} \sum\limits_{n=1}^{N-1}||\partial_\tau\bm f^{n+1}||^2+\frac{\tau C_p^2}{k_m}\sum\limits_{n=1}^{N-1}||g^{n+1}||^2+\frac{\tau C_p^2}{\theta_m}\sum\limits_{n=1}^{N-1}||z^{n+1}||^2.\label{ssp11}
\end{align}
For $E^1$, by the energy estimate of \eqref{b7}-\eqref{b9} at $t_1$, we obtain
$E^1\le E^0.$ By Gronwall inequality, the proof is completed.
\end{proof}

\begin{remark}\label{remark2}
In Theorem \ref{thm0}, we can determine the value of the stabilization parameter $\gamma$ in \eqref{s3} of the sequential algorithm. For instance, we set $\gamma = 2\left(\frac{\alpha^2}{k_m} + \frac{\beta^2}{\theta_m}\right),$
which is the minimum value that guarantees $\frac{\gamma}{2} - \left(\frac{\alpha^2}{k_m} + \frac{\beta^2}{\theta_m}\right) \ge 0$ in the theoretical derivation from \eqref{ssp9} to \eqref{ssp10}. This is also the difference from the spatial discretization method that uses discontinuous Galerkin method. And this stabilization parameter $\gamma$ is independent of the mesh and time step sizes, and only depends on the material parameters, which is similar to the stabilization parameter in \cite{MR3777103}. The numerical tests in Section \ref{eg2} show that in the practical computation, our proposed method is not sensitive to small values of $\gamma$. Then the theoretical relationship given for $\gamma$ can be further relaxed.
\end{remark}

\subsection{A priori error estimate}\label{ee}
In this subsection, we first state several inequalities which will be used. Then, we present the error equations and establish optimal-order convergence bounds in both spatial and temporal discretizations.

Let
\begin{align*}
&{\bm\rho}_u(t)={\mathbf u}(t)-R_u{\mathbf u}(t),\ {\bm\theta}_u(t)=R_u{\mathbf u}(t)-{\mathbf u}_h(t),\\
&\rho_p(t)=p(t)-R_pp(t),\ \theta_p(t)=R_pp(t)-p_h(t),\\
&\rho_T(t)=T(t)-R_TT(t),\ \theta_T(t)=R_TT(t)-T_h(t).
\end{align*}

The following lemmas concerning the elliptic projections' approximation capabilities are stated.
\begin{lemma}\label{lem3.7}\cite{MR4432106} For ${\mathbf u}\in (H^{k_1+1}(\Omega))^d$, $p\in H^{k_2+1}(\Omega)$, $T\in H^{k_3+1}(\Omega)$, the following estimates hold
\begin{align*}
\|{\bm\rho}_u\|+h||\bm\epsilon({\bm\rho}_u)||&\leq Ch^{k_1+1}\|{\mathbf u}\|_{k_1+1},\\
\|{\rho}_p\|+h||\nabla{\rho}_p||&\leq Ch^{k_2+1}\|p\|_{k_2+1},\\
\|{\rho}_T\|+h||\nabla{\rho}_T||&\leq Ch^{k_3+1}\|p\|_{k_3+1}.
\end{align*}
\end{lemma}
\begin{lemma}\label{lem3.8}\cite{MR3777103}The following inequalities are true for ${\mathbf u}\in (H_0^{1}(\Omega))^d$, $p\in H_0^{1}(\Omega)$, and $T\in H_0^{1}(\Omega)$,
\begin{align*}
\|{R}_u{\mathbf u}\|\leq C\|{\mathbf u}\|_{1},\quad \|R_pp\|\leq C\|p\|_{1},\quad \|R_TT\|\leq C\|T\|_{1}.
\end{align*}
\end{lemma}

Also, notice that
\begin{align}
\|\phi_t^{n+1}-\partial_\tau \phi^{n+1}\|^2
&\leq C\tau\int_{t_{n}}^{t_{n+1}}\|\phi_{tt}\|^2dt,\label{pt}\\
||\phi^{n+1}-\phi^{n}||^2&\le C\tau\int_{t_n}^{t_{n+1}}||\phi_t||^2dt.\label{pt0}
\end{align}

The error equations can be obtained by subtracting \eqref{s1} from \eqref{b2}, \eqref{s2} from \eqref{b3} and \eqref{s3} from \eqref{b1}, along with \eqref{ea1}, \eqref{ea2} and \eqref{ea3},
\begin{align}
&c_0(\partial_\tau \theta_p^{n+1},q_h)-b_0(\partial_\tau \theta_T^n,q_h)-\alpha b(q_h,\partial_\tau{\bm \theta}_u^n)+c_p(\theta_p^{n+1},q_h)\notag\\
=&-c_0(p_t^{n+1}-\partial_\tau p^{n+1},q_h)-c_0(\partial_\tau\rho_p^{n+1},q_h)+b_0(T_t^{n+1}-T_t^n,q_h)+b_0(T_t^n-\partial_\tau T^n,q_h)\notag\\
&+b_0(\partial_\tau\rho_T^n,q_h)+\alpha b(q_h,{\mathbf u}_t^{n+1}-{\mathbf u}_t^n)+\alpha b(q_h,{\mathbf u}_t^n-\partial_\tau {\mathbf u}^n)+\alpha b(q_h,\partial_\tau\bm\rho_u^n),\label{ea4}\\
&a_0(\partial_\tau \theta_T^{n+1},s_h)-b_0(\partial_\tau \theta_p^{n+1},s_h)-\beta b(s_h,\partial_\tau{\bm \theta}_u^n)+c_T(\theta_T^{n+1},s_h)\notag\\
=&-a_0(T_t^{n+1}-\partial_\tau T^{n+1},s_h)-a_0(\partial_\tau\rho_T^{n+1},s_h)+b_0(p_t^{n+1}-\partial_\tau p^{n+1},s_h)+b_0(\partial_\tau\rho_p^{n+1},s_h)\notag\\
&+\beta b(s_h,{\mathbf u}_t^{n+1}-{\mathbf u}_t^n)+\beta b(s_h,{\mathbf u}_t^{n}-\partial_\tau {\mathbf u}^n)+\beta b(s_h,\partial_\tau\bm\rho_u^n)+\delta(\rho_T^{n+1},s_h)\notag\\
&+(\mathcal M({\mathbf K}\nabla p_h^{n})\cdot\nabla \theta_T^{n+1},s_h)+(({\mathbf K}\nabla p^{n+1}-\mathcal M({\mathbf K}\nabla p_h^{n}))\cdot\nabla R_TT^{n+1},s_h),\label{ea5}\\
&a({\bm \theta}_u^{n+1},{\mathbf v_h} )+\alpha b(\theta_p^{n+1},{\mathbf v_h})+\beta b(\theta_T^{n+1},{\mathbf v_h})+\gamma(\partial_\tau{\bm\theta}_u^{n+1}-\partial_\tau{\bm\theta}_u^{n},{\mathbf v_h})\notag\\
=&-\alpha b(\rho_p^{n+1},{\mathbf v_h})-\beta b(\rho_T^{n+1},{\mathbf v_h})+\gamma(\partial_\tau R_u{\mathbf u}^{n+1}-\partial_\tau R_u{\mathbf u}^{n},{\mathbf v_h}).\label{ea6}
\end{align}

The optimal convergence order estimate is now prepared as follows.
\begin{theorem}\label{thm1}
Suppose that the exact solutions $({\mathbf u},p,T)$ of problem \eqref{b1}-\eqref{b3} satisfy Assumption \ref{js} with $k_1$, $k_2$, $k_3\ge1$. Let $({\mathbf u}_h^n,p_h^n,T_h^n)\in {\mathbf V}_h\times Q_h\times S_h$ be the corresponding discrete solutions generated by the sequential scheme \eqref{s1}-\eqref{s3}, where the stabilization parameter is set to be $\gamma=2(\frac{\alpha^2}{k_m}+\frac{\beta^2}{\theta_m})$. Then, for $n\geq1$, we obtain the following estimate
\begin{align*}
||\mathbf u^n-\mathbf u_h^n||_1+||p^{n}-p_h^n||+||T^{n}-T_h^n||
\le C(\tau+h^{k_1}+h^{k_2+1}+h^{k_3+1}).
\end{align*}
\end{theorem}
\begin{proof}
Selecting ${\mathbf v}_h=\partial_\tau{\bm\theta}_u^{n+1}$ in \eqref{ea4}, $q_h=\theta_p^{n+1}$ in \eqref{ea5} and $s_h=\theta_T^{n+1}$ in \eqref{ea6}, summing these equations results in
\begin{align}
&c_0(\partial_\tau \theta_p^{n+1},\theta_p^{n+1})+a_0(\partial_\tau \theta_T^{n+1}, \theta_T^{n+1})+a({\bm\theta}_u^{n+1},{\partial_\tau\bm\theta}_u^{n+1})\notag\\
&+c_p(\theta_p^{n+1},\theta_p^{n+1}) +c_T(\theta_T^{n+1},\theta_T^{n+1})+\gamma({\partial_\tau\bm\theta}_u^{n+1}-{\partial_\tau\bm\theta}_u^n,{\partial_\tau\bm\theta}_u^{n+1})\notag\\
=  &b_0(\partial_\tau \theta_T^{n}, \theta_p^{n+1})+b_0(\partial_\tau\theta_p^{n+1}, \theta_T^{n+1})-c_0(p_t^{n+1}-\partial_\tau p^{n+1}, \theta_p^{n+1})-c_0(\partial_\tau\rho_p^{n+1},\theta_p^{n+1})\notag\\
&+b_0(T_t^{n+1}-T_t^n, \theta_p^{n+1})+b_0(T_t^{n}-\partial_\tau T^{n}, \theta_p^{n+1})+b_0(\partial_\tau\rho_T^n,\theta_p^{n+1})-a_0(T_t^{n+1}-\partial_\tau T^{n+1},\theta_T^{n+1})\notag\\
&-a_0(\partial_\tau\rho_T^{n+1},\theta_T^{n+1})+b_0(p_t^{n+1}-\partial_\tau p^{n+1},\theta_T^{n+1})+b_0(\partial_\tau \rho_p^{n+1},\theta_T^{n+1})\notag\\
&+\alpha b(\theta_p^{n+1},{\mathbf u}_t^{n+1}-{\mathbf u}_t^n)+\alpha b(\theta_p^{n+1},{\mathbf u}_t^{n}-\partial_\tau {\mathbf u}^{n})+\alpha b(\theta_p^{n+1},\partial_\tau\bm\rho_u^{n})\notag\\
&+\beta b(\theta_T^{n+1},{\mathbf u}_t^{n+1}-{\mathbf u}_t^n)+\beta b(\theta_T^{n+1},{\mathbf u}_t^n-\partial_\tau {\mathbf u}^n)+\beta b(\theta_T^{n+1},\partial_\tau\bm\rho_u^n)\notag\\
&+(\mathcal M({\mathbf K}\nabla p_h^{n})\cdot\nabla \theta_T^{n+1},\theta_T^{n+1})+(({\mathbf K}\nabla p^{n+1}-\mathcal M({\mathbf K}\nabla p_h^{n}))\cdot\nabla R_TT^{n+1},\theta_T^{n+1})\notag\\
&-\alpha(b(\theta_p^{n+1},\partial_\tau\bm\theta_u^{n+1})-b(\theta_p^{n+1},\partial_\tau\bm\theta_u^n))-\beta(b(\theta_T^{n+1},\partial_\tau\bm\theta_u^{n+1})-b(\theta_T^{n+1},\partial_\tau\bm\theta_u^n))\notag\\
&-\alpha b(\rho_p^{n+1},\partial_\tau\bm\theta_u^{n+1})-\beta b(\rho_T^{n+1},\partial_\tau\bm\theta_u^{n+1})+\gamma(\partial_\tau R_u{\mathbf u}^{n+1}-\partial_\tau R_u{\mathbf u}^{n},\partial_\tau\bm\theta_u^{n+1})+\delta(\rho_T^{n+1},\theta_T^{n+1}).\label{ea7}
\end{align}
Using Lemma \ref{lem2.4} and the symmetry of $a(\cdot,\cdot)$, we obtain
\begin{align*}
&c_0(\partial_\tau \theta_p^{n+1},\theta_p^{n+1})=\frac{c_0}{2\tau}(||\theta_p^{n+1}||^2-||\theta_p^{n}||^2)+\frac{c_0}{2\tau}||\theta_p^{n+1}-\theta_p^{n}||^2,\\
& a_0(\partial_\tau \theta_T^{n+1}, \theta_T^{n+1})=\frac{a_0}{2\tau}(||\theta_T^{n+1}||^2-||\theta_T^{n}||^2)+\frac{a_0}{2\tau}||\theta_T^{n+1}-\theta_T^{n}||^2,\\
& a({\bm\theta}_u^{n+1},{\partial_\tau\bm\theta}_u^{n+1})=\frac{1}{2\tau}( a({\bm\theta}_u^{n+1},{\bm\theta}_u^{n+1})-a({\bm\theta}_u^{n},{\bm\theta}_u^{n}))+\frac{1}{2\tau} a({\bm\theta}_u^{n+1}-{\bm\theta}_u^{n},{\bm\theta}_u^{n+1}-{\bm\theta}_u^{n}),\\
&c_p(\theta_p^{n+1},\theta_p^{n+1})\ge k_m||\nabla\theta_p^{n+1} ||^2, c_T(\theta_T^{n+1},\theta_T^{n+1})\ge \theta_m||\theta_T^{n+1} ||^2,\\
&\gamma({\partial_\tau\bm\theta}_u^{n+1},{\partial_\tau\bm\theta}_u^{n+1})-\gamma({\partial_\tau\bm\theta}_u^n,{\partial_\tau\bm\theta}_u^{n+1}) =\frac{\gamma}{2}(||{\partial_\tau\bm\theta}_u^{n+1}||^2-||{\partial_\tau\bm\theta}_u^{n}||^2)+\frac{\gamma}{2}||{\partial_\tau\bm\theta}_u^{n+1}-{\partial_\tau\bm\theta}_u^{n}||^2.
\end{align*}
Also, notice that
\begin{align*}
&b_0(\partial_\tau \theta_T^n,\theta_p^{n+1})+b_0(\partial_\tau\theta_p^{n+1},\theta_T^{n+1})\\
=&b_0\tau(\partial_\tau \theta_T^{n+1},\partial_\tau\theta_p^{n+1})+b_0\partial_\tau(\theta_p^{n+1},\theta_T^{n+1})+b_0(\partial_\tau \theta_T^{n}-\partial_\tau \theta_T^{n+1},\theta_p^{n+1})\\
\le&\frac{b_0}{2}\tau||\partial_\tau \theta_T^{n+1}||^2+\frac{b_0}{2}\tau||\partial_\tau \theta_p^{n+1}||^2+b_0\partial_\tau(\theta_p^{n+1},\theta_T^{n+1})+b_0(\partial_\tau \theta_T^{n}-\partial_\tau \theta_T^{n+1},\theta_p^{n+1}).
\end{align*}
Multiplying the both sides of \eqref{ea7} by $2\tau$, summing up $n$ from 1 to $N-1$ and utilizing the equations shown above, we provide
\begin{align}\label{z11}
&c_0(||\theta_p^{N}||^2-||\theta_p^{1}||^2)+(c_0-b_0)\tau^2\sum\limits_{n=1}^{N-1}||\partial_\tau\theta_p^{n+1}||^2+a_0(||\theta_T^{N}||^2-||\theta_T^{1}||^2)\notag\\
&+(a_0-b_0)\tau^2\sum\limits_{n=1}^{N-1}||\partial_\tau\theta_T^{n+1}||^2+a({\bm\theta}_u^{N},{\bm\theta}_u^{N})-a({\bm\theta}_u^{1},{\bm\theta}_u^{1})+\tau^2\sum\limits_{n=1}^{N-1}a(\partial_\tau{\bm\theta}_u^{n+1},\partial_\tau{\bm\theta}_u^{n+1}) \notag\\
&+2\tau\sum\limits_{n=1}^{N-1}(k_m||\nabla\theta_p^{n+1} ||^2+\theta_m||\nabla\theta_T^{n+1}||^2)+\gamma\tau(||{\partial_\tau\bm\theta}_u^{N}||^2-||{\partial_\tau\bm\theta}_u^{1}||^2)+\gamma\tau\sum\limits_{n=1}^{N-1}||{\partial_\tau\bm\theta}_u^{n+1}-{\partial_\tau\bm\theta}_u^{n}||^2\notag\\
\le&I_1+\cdots+I_{16},
\end{align}
where
\begin{align*}
I_1=&2b_0((\theta_p^{N}, \theta_T^{N})-(\theta_p^{1}, \theta_T^{1})),\\
I_2=&2b_0\tau\sum\limits_{n=1}^{N-1}(\partial_\tau\theta_T^{n}-\partial_\tau\theta_T^{n+1}, \theta_p^{n+1}),\\
I_3=&-2c_0\tau\sum\limits_{n=1}^{N-1}[(p_t^{n+1}-\partial_\tau p^{n+1}, \theta_p^{n+1})+(\partial_\tau\rho_p^{n+1},\theta_p^{n+1})],\\
I_4= &2b_0\tau\sum\limits_{n=1}^{N-1}[(T_t^{n+1}-T_t^n, \theta_p^{n+1})+(T_t^{n}-\partial_\tau T^{n}, \theta_p^{n+1})+(\partial_\tau\rho_T^n,\theta_p^{n+1})],\\
I_5= &-2a_0\tau\sum\limits_{n=1}^{N-1}[(T_t^{n+1}-\partial_\tau T^{n+1},\theta_T^{n+1})+(\partial_\tau\rho_T^{n+1},\theta_T^{n+1})],\\
I_6= &2b_0\tau\sum\limits_{n=1}^{N-1}[(p_t^{n+1}-\partial_\tau p^{n+1},\theta_T^{n+1})+(\partial_\tau \rho_p^{n+1},\theta_T^{n+1})],\\
I_7= &2\alpha\tau\sum\limits_{n=1}^{N-1}[ b(\theta_p^{n+1},{\mathbf u}_t^{n+1}-{\mathbf u}_t^n)+ b(\theta_p^{n+1},{\mathbf u}_t^{n}-\partial_\tau {\mathbf u}^{n})+ b(\theta_p^{n+1},\partial_\tau\bm\rho_u^{n})],\\
I_8= & 2\beta\tau\sum\limits_{n=1}^{N-1}[ b(\theta_T^{n+1},{\mathbf u}_t^{n+1}-{\mathbf u}_t^n)+ b(\theta_T^{n+1},{\mathbf u}_t^n-\partial_\tau {\mathbf u}^n)+ b(\theta_T^{n+1},\partial_\tau\bm\rho_u^n)],\\
I_9= &2\tau\sum\limits_{n=1}^{N-1}(\mathcal M({\mathbf K}\nabla p_h^{n})\cdot\nabla \theta_T^{n+1},\theta_T^{n+1}),\\
I_{10}=&2\tau\sum\limits_{n=1}^{N-1}(({\mathbf K}\nabla p^{n+1}-\mathcal M({\mathbf K}\nabla p_h^{n}))\cdot\nabla R_TT^{n+1},\theta_T^{n+1}),\\
 I_{11}= &-2\alpha\tau\sum\limits_{n=1}^{N-1}(b(\theta_p^{n+1},\partial_\tau\bm\theta_u^{n+1})-b(\theta_p^{n+1},\partial_\tau\bm\theta_u^n)),\\
I_{12}= &-2\beta\tau\sum\limits_{n=1}^{N-1}(b(\theta_T^{n+1},\partial_\tau\bm\theta_u^{n+1})-b(\theta_T^{n+1},\partial_\tau\bm\theta_u^n)),\\
I_{13}= &-2\alpha\tau\sum\limits_{n=1}^{N-1} b(\rho_p^{n+1},\partial_\tau\bm\theta_u^{n+1}),\\
I_{14}= &-2\beta\tau\sum\limits_{n=1}^{N-1}b(\rho_T^{n+1},\partial_\tau\bm\theta_u^{n+1}),\\
I_{15}= &2\gamma\tau\sum\limits_{n=1}^{N-1}(\partial_\tau R_u{\mathbf u}^{n+1}-\partial_\tau R_u{\mathbf u}^{n},\partial_\tau\bm\theta_u^{n+1}),\\
I_{16}=&2\delta\tau\sum\limits_{n=1}^{N-1}(\rho_T^{n+1},\theta_T^{n+1}).
\end{align*}

The next step is to determine upper bounds for the terms $I_1$-$I_{15}$ on the right-hand side of \eqref{z11}. Using the Young's inequality and Cauchy-Schwarz inequality, we get
\begin{align*}
|I_1|\le b_0||\theta_p^{N}||^2+b_0||\theta_T^{N}||^2+b_0||\theta_p^{1}||^2+b_0||\theta_T^{1}||^2.
\end{align*}
It follows from \eqref{f} that
\begin{align*}
|I_2|
=&2b_0\tau|(\partial_\tau \theta_T^{1},\theta_p^{1})-(\partial_\tau \theta_T^{N},\theta_p^{N})+\sum\limits_{n=1}^{N-1}\tau(\partial_\tau \theta_T^{n},\partial_\tau \theta_p^{n+1})|\\
\le&b_0\tau^2||\partial_\tau \theta_T^{1}||^2+b_0||\theta_p^{1}||^2+b_0\tau^2||\partial_\tau \theta_T^{N}||^2+b_0||\theta_p^{N}||^2\\
&+b_0\tau \sum\limits_{n=1}^{N-1}\tau||\partial_\tau \theta_T^{n}||^2+b_0\tau \sum\limits_{n=1}^{N-1}\tau||\partial_\tau \theta_p^{n+1}||^2.
\end{align*}
The term $I_3$ can be constrained by using \eqref{pt}, the Young's inequality and Cauchy-Schwartz inequality,
\begin{align*}
|I_3|
&\leq 2c_0\tau\sum\limits_{n=1}^{N-1}||\theta_p^{n+1} ||^2+c_0\tau\sum\limits_{n=1}^{N-1}\|p_t^{n+1}-\partial_\tau p^{n+1}\|^2 +c_0\tau\sum\limits_{n=1}^{N-1}||\partial_\tau\rho_p^{n+1}||^2\\
&\leq 2c_0\tau\sum\limits_{n=1}^{N-1}||\theta_p^{n+1} ||^2+C\tau^2\int_{t_1}^{t_{N}}\|p_{tt}\|^2dt +Ch^{2k_2+2}\int_{t_1}^{t_{N}}\|p_t\|_{k_2+1}^2dt.
\end{align*}
Similarly,
\begin{align*}
|I_4|
&\leq 3b_0\tau\sum\limits_{n=1}^{N-1}||\theta_p^{n+1} ||^2+b_0\tau\sum\limits_{n=1}^{N-1}\|T_t^{n+1}-T_t^{n}\|^2 +b_0\tau\sum\limits_{n=1}^{N-1}\|T_t^{n}-\partial_\tau T^{n}\|^2 +b_0\tau\sum\limits_{n=1}^{N-1}||\partial_\tau\rho_T^{n}||^2\\
&\leq 3b_0\tau\sum\limits_{n=1}^{N-1}||\theta_p^{n+1} ||^2+C\tau^2\int_{t_{0}}^{t_{N-1}}(\|T_{t}\|^2+\|T_{tt}\|^2)dt +Ch^{2k_3+2}\int_{t_{0}}^{t_{N-1}}\|T_t\|_{k_3+1}^2dt,\\
|I_5|
&\leq 2a_0\tau\sum\limits_{n=1}^{N-1}||\theta_T^{n+1} ||^2+C\tau^2\int_{t_{1}}^{t_{N}}\|T_{tt}\|^2dt +Ch^{2k_3+2}\int_{t_{1}}^{t_{N}}\|T_t\|_{k_3+1}^2dt,\\
|I_6|
&\leq 2b_0\tau\sum\limits_{n=1}^{N-1}||\theta_T^{n+1} ||^2+C\tau^2\int_{t_1}^{t_{N}}\|p_{tt}\|^2dt +Ch^{2k_2+2}\int_{t_1}^{t_{N}}\|p_t\|_{k_2+1}^2dt.
\end{align*}
For $I_7$, it follows from the Cauchy-Schwartz inequality and Young's inequality that,
\begin{align*}
&\alpha b(\theta_p^{n+1},{\mathbf u}_t^{n+1}-{\mathbf u}_t^n)
\le \frac{k_m}{24}||\nabla\theta_p^{n+1} ||^2+C||{\mathbf u}_t^{n+1}-{\mathbf u}_t^n||^2,\\
&\alpha b(\theta_p^{n+1},{\mathbf u}_t^{n}-\partial_\tau {\mathbf u}^n)
\le \frac{k_m}{24}||\nabla\theta_p^{n+1} ||^2+C||{\mathbf u}_t^{n}-\partial_\tau {\mathbf u}^n||^2,\\
&\alpha b(\theta_p^{n+1},\partial_\tau\bm\rho_u^n)
\le \frac{k_m}{24}||\nabla\theta_p^{n+1} ||^2+C||\partial_\tau\bm\rho_u^n||^2.
\end{align*}
Combined with \eqref{pt} and \eqref{pt0}, we provide
\begin{align*}
|I_7|
&\leq \frac{k_m}{4}\tau\sum\limits_{n=1}^{N-1}||\nabla\theta_p^{n+1} ||^2+C(\tau^2\int_{t_{0}}^{t_{N}}||{\mathbf u}_{tt}||^2dt+h^{2k_1+2}\int_{t_{0}}^{t_{N-1}}\|{\mathbf u}_t\|_{k_1+1}^2dt).
\end{align*}
Similarly,
 \begin{align*}
|I_8|
&\leq \frac{\theta_m}{4}\tau\sum\limits_{n=1}^{N-1}||\nabla\theta_T^{n+1} ||^2+C\left(\tau^2\int_{t_{0}}^{t_{N}}||{\mathbf u}_{tt}||_V^2dt+h^{2k_1+2}\int_{t_{0}}^{t_{N-1}}\|{\mathbf u}_t\|_{k_1+1}^2dt\right).
\end{align*}
Applying Lemma \ref{lem0}, Young's inequality and Cauchy-Schwartz inequality, we render
\begin{align*}
|I_9|\leq \frac{\theta_m}{8}\tau\sum\limits_{n=1}^{N-1}||\nabla\theta_T^{n+1}||^2+C\tau\sum\limits_{n=1}^{N-1}\|\theta_T^{n+1}\|^2.
\end{align*}
Due to the suitably cut-off operator constant $M$ \cite{MR2116915}, for $I_{10}$, we obtain
\begin{align*}
&(({\mathbf K}\nabla p^{n+1}-\mathcal M({\mathbf K}\nabla p_h^{n}))\cdot\nabla R_TT^{n+1},\theta_T^{n+1})\\
=&({\mathbf K}\nabla\rho_p^{n+1}\cdot\nabla R_TT^{n+1},\theta_T^{n+1})+({\mathbf K}\nabla(R_pp^{n+1}-R_pp^{n})\cdot\nabla R_TT^{n+1},\theta_T^{n+1})\\
&+((\mathcal M({\mathbf K}\nabla R_pp^{n})-\mathcal M({\mathbf K}\nabla p_h^{n}))\cdot\nabla R_TT^{n+1},\theta_T^{n+1})\\
:=&H_1+H_2+H_3.
\end{align*}
For $H_1$, we have
\begin{align*}
|H_1|\le& k_M|(\nabla T^{n+1}\nabla\rho_p^{n+1},\theta_T^{n+1})-(\nabla \rho_T^{n+1}\cdot\nabla\rho_p^{n+1},\theta_T^{n+1})|\notag\\
=&k_M|\langle\rho_p^{n+1},\nabla T^{n+1}\cdot\theta_T^{n+1}\bm n\rangle_{\partial\Omega}-(\rho_p^{n+1},\nabla T^{n+1}\cdot\nabla\theta_T^{n+1})\notag\\
&-(\rho_p^{n+1},\nabla\cdot(\nabla T^{n+1})\cdot\theta_T^{n+1})-(\nabla \rho_T^{n+1}\cdot\nabla\rho_p^{n+1},\theta_T^{n+1})|\notag\\
=&k_M|(\rho_p^{n+1},\nabla T^{n+1}\cdot\nabla\theta_T^{n+1})+(\rho_p^{n+1},\nabla\cdot(\nabla T^{n+1})\cdot\theta_T^{n+1})+(\nabla \rho_T^{n+1}\cdot\nabla\rho_p^{n+1},\theta_T^{n+1})|.
\end{align*}
Combining the Young's inequality, Cauchy-Schwartz inequality and \eqref{pt0}, we obtain
\begin{align*}
(\rho_p^{n+1},\nabla T^{n+1}\cdot\nabla\theta_T^{n+1})\le& \frac{\theta_m}{8}||\nabla\theta_T^{n+1}||^2+C||\rho_p^{n+1}||^2,\\
(\rho_p^{n+1},\nabla\cdot(\nabla T^{n+1})\cdot\theta_T^{n+1})\le& C||\theta_T^{n+1}||^2+C||\rho_p^{n+1}||^2,\\
(\nabla \rho_T^{n+1}\cdot\nabla\rho_p^{n+1},\theta_T^{n+1})
\le& C\tau\int_{t_{n}}^{t_{n+1}}||\nabla\partial_tR_hp||^2dt+C||\theta_T^{n+1}||^2.
\end{align*}
The $L^{\infty}(\Omega)$-error estimate in \cite{MR717695} provides us with
\begin{equation*}
||\nabla\rho_T^{n+1}||_{0,\infty}\le C||T||_{k_3+1,\infty}h^{k_3}.
\end{equation*}
Therefore,
\begin{align*}
(\nabla \rho_T^{n+1}\cdot\nabla\rho_p^{n+1},\theta_T^{n+1})
&\le\frac{1}{2}||\theta_T^{n+1}||^2+\frac{1}{2}||\nabla \rho_T^{n+1}||_{0,\infty}^2||\nabla\rho_p^{n+1}||^2\\
&\le C||\theta_T^{n+1}||^2+Ch^{2k_2}h^{2k_3}\\
&\le C||\theta_T^{n+1}||^2+C(h^{2k_2+2}+h^{2k_3+2}).
\end{align*}
Then, we have
\begin{align*}
|H_1|\le C(h^{2k_2+2}+h^{2k_3+2})+\frac{\theta_m}{8}||\nabla\theta_T^{n+1}||^2+C||\theta_T^{n+1}||^2.
\end{align*}
By \eqref{pt0}, we obtain
\begin{align*}
|H_2|\le C||\nabla (R_pp^{n+1}-R_pp^{n})||^2+C||\theta_T^{n+1}||^2\le C\tau\int_{t_{n}}^{t_{n+1}}||\nabla R_pp_t||^2dt+C||\theta_T^{n+1}||^2.
\end{align*}
Observing that  $|\mathcal M({\mathbf K}\nabla p^{n+1})-\mathcal M({\mathbf K}\nabla p_h^{n})|\le|{\mathbf K}\nabla p^{n+1}-{\mathbf K}\nabla p_h^{n}|$, we get
\begin{align*}
|H_3|\le||\nabla R_TT^{n+1}||_{0,\infty} ||{\mathbf K}\nabla R_p p^{n}-{\mathbf K}\nabla p_h^{n}||||\theta_T^{n+1}||\le \frac{k_m}{4}||\nabla\theta_p^{n}||^2+C||\theta_T^{n+1}||^2.
\end{align*}
By combining $H_1$-$H_3$, we have
\begin{align*}
&(({\mathbf K}\nabla p^{n+1}-\mathcal M({\mathbf K}\nabla p_h^{n}))\cdot\nabla R_TT^{n+1},\theta_T^{n+1})\\
\leq & C(h^{2k_2+2}+h^{2k_3+2})+C\tau\int_{t_{n}}^{t_{n+1}}||\nabla p_t||_1^2dt+\frac{k_m}{4}\|\nabla\theta_p^{n}\|^2+\frac{\theta_m}{8}\|\nabla\theta_T^{n+1}\|^2+C||\theta_T^{n+1}||^2.
\end{align*}
Thus,
\begin{align*}
|I_{10}|
\leq &CT(h^{2k_2+2}+h^{2k_3+2})+C\tau^2\int_{t_{1}}^{t_{N}}||\nabla p_t||_1^2dt\\
&+\frac{k_m}{4}\tau\sum\limits_{n=1}^{N-1}\|\nabla\theta_p^{n}\|^2+\frac{\theta_m}{8}\tau\sum\limits_{n=1}^{N-1}\|\nabla\theta_T^{n+1}\|^2+C\tau\sum\limits_{n=1}^{N-1}||\theta_T^{n+1}||^2.
\end{align*}
Similarly to $I_7$, estimated $I_{11}$ and $I_{12}$,
\begin{align*}
\alpha(b(\theta_p^{n+1},\partial_\tau\bm\theta_u^{n+1})-b(\theta_p^{n+1},\partial_\tau\bm\theta_u^n))
=&\alpha\sum\limits_{K\in \mathcal T_h}\int_K (\partial_\tau\bm\theta_u^{n+1}-\partial_\tau\bm\theta_u^n)\cdot\nabla\theta_p^{n+1}dK\\
\le& \frac{k_m}{2}||\nabla\theta_p^{n+1}||^2+\frac{\alpha^2}{2k_m}||\partial_\tau\bm\theta_u^{n+1}-\partial_\tau\bm\theta_u^n||^2,
\end{align*}
we obtain
\begin{align*}
|I_{11}|+|I_{12}|\le k_m\tau\sum\limits_{n=1}^{N-1}||\nabla\theta_p^{n+1}||^2+ \theta_m\tau\sum\limits_{n=1}^{N-1}||\nabla\theta_T^{n+1}||^2+(\frac{\alpha^2}{k_m}+\frac{\beta^2}{\theta_m})\tau\sum\limits_{n=1}^{N-1}||\partial_\tau\bm\theta_u^{n+1}-\partial_\tau\bm\theta_u^n||^2.
\end{align*}
From \eqref{f}, we get
\begin{align*}
I_{13}=-2\alpha b(\rho_p^{N},\bm\theta_u^{N}) + 2\alpha b(\rho_p^{1},\bm\theta_u^{1}) +2\alpha\sum\limits_{n=1}^{N-1}  b(\rho_p^{n+1}-\rho_p^{n},\bm\theta_u^{n}).
\end{align*}
Applying the Green's theorem yields that
\begin{align*}
\alpha b(\rho_p^{N},\bm\theta_u^{N})=&\alpha\sum\limits_{K\in \mathcal T_h}\int_K \rho_p^{N}\nabla\cdot \bm\theta_u^{N}dK\le \frac{\mu}{3}||\bm\epsilon(\bm\theta_u^{N})||^2+Ch^{2k_2+2},\\
\alpha b(\rho_p^{1},\bm\theta_u^{1})\le& \alpha||\bm\epsilon(\bm\theta_u^{1})||^2+Ch^{2k_2+2},\\
\alpha \sum\limits_{n=1}^{N-1} b(\rho_p^{n+1}-\rho_p^{n},\bm\theta_u^{n})
\le&\alpha \tau\sum\limits_{n=1}^{N-1} ||\bm\epsilon(\bm\theta_u^{n})||^2+\frac{\alpha C}{\tau}\sum\limits_{n=1}^{N-1}||\rho_p^{n+1}-\rho_p^{n}||^2\\
\le& C\tau\sum\limits_{n=1}^{N-1}||\bm\epsilon(\bm\theta_u^{n})||^2+Ch^{2k_2+2}.
\end{align*}
Hence,
\begin{align*}
|I_{13}|
\leq \frac{\mu}{3}||\bm\epsilon(\bm\theta_u^{N})||^2+\alpha||\bm\epsilon(\bm\theta_u^{1})||^2+ C\tau\sum\limits_{n=1}^{N-1} ||\bm\epsilon(\bm\theta_u^{n})||^2+Ch^{2k_2+2}.
\end{align*}
Similarly,
\begin{align*}
|I_{14}|
\leq \frac{\mu}{3}||\bm\epsilon(\bm\theta_u^{N})||^2+\alpha||\bm\epsilon(\bm\theta_u^{1})||^2+ C\tau\sum\limits_{n=1}^{N-1} ||\bm\epsilon(\bm\theta_u^{n})||^2+Ch^{2k_3+2}.
\end{align*}
Furthermore, utilizing \eqref{f}, we get
\begin{align*}
I_{15}=&\gamma(\frac{R_u{\mathbf u}^{2}-R_u{\mathbf u}^{1}}{\tau}-\frac{R_u{\mathbf u}^{1}-R_u{\mathbf u}^{0}}{\tau_0},\bm\theta_u^{2}-\bm\theta_u^{1})+\gamma(\frac{R_u{\mathbf u}^{3}-R_u{\mathbf u}^{2}}{\tau}-\frac{R_u{\mathbf u}^{2}-R_u{\mathbf u}^{1}}{\tau},\bm\theta_u^{3}-\bm\theta_u^{2})\\
&+\sum\limits_{n=3}^{N-1}\gamma(\frac{R_u{\mathbf u}^{n+1}-R_u{\mathbf u}^{n}}{\tau}-\frac{R_u{\mathbf u}^{n}-R_u{\mathbf u}^{n-1}}{\tau},\bm\theta_u^{n+1}-\bm\theta_u^{n})\\
=&\gamma(\frac{R_u{\mathbf u}^{2}-R_u{\mathbf u}^{1}}{\tau}-\frac{R_u{\mathbf u}^{1}-R_u{\mathbf u}^{0}}{\tau_0},\bm\theta_u^{2}-\bm\theta_u^{1})+\gamma(\frac{R_u{\mathbf u}^{3}-R_u{\mathbf u}^{2}}{\tau}-\frac{R_u{\mathbf u}^{2}-R_u{\mathbf u}^{1}}{\tau},\bm\theta_u^{3}-\bm\theta_u^{2})\\
&+\gamma(\frac{R_u{\mathbf u}^{N}-R_u{\mathbf u}^{N-1}}{\tau}-\frac{R_u{\mathbf u}^{N-1}-R_u{\mathbf u}^{N-2}}{\tau},\bm\theta_u^{N})-\gamma(\frac{R_u{\mathbf u}^{3}-R_u{\mathbf u}^{2}}{\tau}-\frac{R_u{\mathbf u}^{2}-R_u{\mathbf u}^{1}}{\tau},\bm\theta_u^{3})\\
&-\sum\limits_{n=3}^{N-1}\gamma(\frac{R_u{\mathbf u}^{n+1}-2R_u{\mathbf u}^{n}+R_u{\mathbf u}^{n-1}}{\tau}-\frac{R_u{\mathbf u}^{n}-2R_u{\mathbf u}^{n-1}+R_u{\mathbf u}^{n-2}}{\tau},\bm\theta_u^{n})\\
:=&J_1+J_2+J_3+J_4+J_5.
\end{align*}
It follows from the Young's inequality and Cauchy-Schwarz inequality that
\begin{align*}
J_1=&\gamma(\frac{R_u{\mathbf u}^{2}-R_u{\mathbf u}^{1}}{\tau}-\frac{R_u{\mathbf u}^{1}-R_u{\mathbf u}^{0}}{\tau_0},\bm\theta_u^{2}-\bm\theta_u^{1})\\
=&\gamma(\frac{1}{\tau}\int_{t_{1}}^{t_{2}}(t_2-t)R_u{\mathbf u}_{tt}dt+\frac{1}{\tau}\int_{t_{0}}^{t_{1}}(t-t_0)R_u{\mathbf u}_{tt}dt,\bm\theta_u^{2}-\bm\theta_u^{1})\\
\le&\gamma||\int_{t_{0}}^{t_{2}}R_u{\mathbf u}_{tt}dt||||\bm\theta_u^{2}-\bm\theta_u^{1}||\le\frac{\gamma}{4\tau}||\bm\theta_u^{2}-\bm\theta_u^{1}||^2+C\tau^2\int_{t_{0}}^{t_{2}}||R_u{\mathbf u}_{tt}||^2dt,\\
J_2=&\gamma(\frac{R_u{\mathbf u}^{3}-R_u{\mathbf u}^{2}}{\tau}-\frac{R_u{\mathbf u}^{2}-R_u{\mathbf u}^{1}}{\tau},\bm\theta_u^{3}-\bm\theta_u^{2})\le\frac{\gamma}{4\tau}||\bm\theta_u^{3}-\bm\theta_u^{2}||^2+C\tau^2\int_{t_{1}}^{t_{3}}||R_u{\mathbf u}_{tt}||^2dt,\\
J_3=&\gamma(\frac{R_u{\mathbf u}^{N}-R_u{\mathbf u}^{N-1}}{\tau}-\frac{R_u{\mathbf u}^{N-1}-R_u{\mathbf u}^{N-2}}{\tau},\bm\theta_u^{N})\\
=&\gamma(\frac{1}{\tau}\int_{t_{N-1}}^{t_{N}}(t_N-t)R_u{\mathbf u}_{tt}dt+\frac{1}{\tau}\int_{t_{N-2}}^{t_{N-1}}(t-t_{N-2})R_u{\mathbf u}_{tt}dt,\bm\theta_u^{N})\\
\le&\gamma||\int_{t_{N-2}}^{t_{N}}R_u{\mathbf u}_{tt}dt||||\bm\theta_u^{N}||\le C_k\tau(\max\limits_{0\le t\le t_f}||R_u{\mathbf u}_{tt}||)||\bm\epsilon(\bm\theta_u^{N})||\\
\le&\frac{\mu}{3}||\bm\epsilon(\bm\theta_u^{N})||^2+C\tau^2\max\limits_{0\le t\le t_f}||R_u{\mathbf u}_{tt}||^2,\\
J_4=&-\gamma(\frac{R_u{\mathbf u}^{3}-R_u{\mathbf u}^{2}}{\tau}-\frac{R_u{\mathbf u}^{2}-R_u{\mathbf u}^{1}}{\tau},\bm\theta_u^{3})\le\frac{\mu}{2}||\bm\epsilon(\bm\theta_u^{3})||^2+C\tau^2\max\limits_{0\le t\le t_f}||R_u{\mathbf u}_{tt}||^2.
\end{align*}
Based on the Taylor expansion, we present
\begin{align*}
\frac{R_u{\mathbf u}^{n+1}-2R_u{\mathbf u}^{n}+R_u{\mathbf u}^{n-1}}{\tau}
&=\tau R_u{\mathbf u}_{tt}^{n}+\frac{1}{2\tau}\int_{t_n}^{t_{n+1}}(t-t_{n+1})^2R_u{\mathbf u}_{ttt}dt-\frac{1}{2\tau}\int_{t_{n-1}}^{t_{n}}(t-t_{n-1})^2R_u{\mathbf u}_{ttt}dt,\\
\frac{R_u{\mathbf u}^{n}-2R_u{\mathbf u}^{n-1}+R_u{\mathbf u}^{n-2}}{\tau}
&=\tau R_u{\mathbf u}_{tt}^{n-1}+\frac{1}{2\tau}\int_{t_{n-1}}^{t_{n}}(t-t_{n})^2R_u{\mathbf u}_{ttt}dt-\frac{1}{2\tau}\int_{t_{n-2}}^{t_{n-1}}(t-t_{n-2})^2R_u{\mathbf u}_{ttt}dt.
\end{align*}
Therefore, we get
\begin{align*}
&||\frac{R_u{\mathbf u}^{n+1}-2R_u{\mathbf u}^{n}+R_u{\mathbf u}^{n-1}}{\tau}-\frac{R_u{\mathbf u}^{n}-2R_u{\mathbf u}^{n-1}+R_u{\mathbf u}^{n-2}}{\tau}||\\
=&||\tau\int_{t_{n-1}}^{t_n}R_u{\mathbf u}_{ttt}dt+\frac{1}{2\tau}\int_{t_n}^{t_{n+1}}(t-t_{n+1})^2R_u{\mathbf u}_{ttt}dt-\frac{1}{2\tau}\int_{t_{n-1}}^{t_{n}}(t-t_{n-1})^2R_u{\mathbf u}_{ttt}dt\\
&-\frac{1}{2\tau}\int_{t_{n-1}}^{t_{n}}(t-t_{n})^2R_u{\mathbf u}_{ttt}dt+\frac{1}{2\tau}\int_{t_{n-2}}^{t_{n-1}}(t-t_{n-2})^2R_u{\mathbf u}_{ttt}dt||\\
\le&2\tau\int_{t_{n-2}}^{t_{n+1}}||R_u{\mathbf u}_{ttt}||dt.
\end{align*}
It follows from the Young's inequality and Cauchy-Schwarz inequality that
\begin{align*}
J_5=&-\sum\limits_{n=3}^{N-1}\gamma(\frac{R_u{\mathbf u}^{n+1}-2R_u{\mathbf u}^{n}+R_u{\mathbf u}^{n-1}}{\tau}-\frac{R_u{\mathbf u}^{n}-2R_u{\mathbf u}^{n-1}+R_u{\mathbf u}^{n-2}}{\tau},\bm\theta_u^{n})\\
\le&\gamma C\tau\sum\limits_{n=3}^{N-1}||\bm\epsilon(\bm\theta_u^{n})||^2+C\tau^2\int_{t_1}^{t_N}||R_u{\mathbf u}_{ttt}||^2dt.
\end{align*}
Combining the estimates from $J_1$-$J_5$, we get
\begin{align*}
|I_{15}|\le&\frac{\mu}{3}||\bm\epsilon(\bm\theta_u^{N})||^2+\frac{\mu}{2}||\bm\epsilon(\bm\theta_u^{3})||^2+\frac{\gamma}{4\tau}||\bm\theta_u^{2}-\bm\theta_u^{1}||^2+\frac{\gamma}{4\tau}||\bm\theta_u^{3}-\bm\theta_u^{2}||^2\\
&+\gamma C\tau\sum\limits_{n=3}^{N-1}||\bm\epsilon(\bm\theta_u^{n})||^2+C\tau^2(\max\limits_{0\le t\le t_f}||R_u{\mathbf u}_{tt}||^2+\int_{t_1}^{t_N}||R_u{\mathbf u}_{ttt}||^2dt).
\end{align*}
For $I_{16}$, we have
\begin{align*}
I_{16}\le&\delta h^{2k_3+2}+\delta\tau\sum\limits_{n=1}^{N-1}||\theta_T^{n+1}||^2.
\end{align*}

Synthesizing $I_1$-$I_{16}$, we obtain
\begin{align}
&\mu||\bm\epsilon(\bm\theta_u^{N})||^2+\lambda||\nabla\cdot\bm\theta_u^{N}||^2+(c_0-2b_0)||\theta_p^{N}||^2+(c_0-2b_0)\tau^2\sum\limits_{n=1}^{N-1}||\partial_\tau\theta_p^{n+1}||^2+(a_0-b_0)||\theta_T^{N}||^2\notag\\
&+(a_0-2b_0)\tau^2\sum\limits_{n=1}^{N-1}||\partial_\tau\theta_T^{n+1}||^2
+ \frac{k_m}{2}\tau\sum\limits_{n=1}^{N-1}||\nabla\theta_p^{n+1} ||^2+\frac{\theta_m}{2}\tau\sum\limits_{n=1}^{N-1}||\nabla\theta_T^{n+1}||^2\notag\\
&+\gamma\tau||{\partial_\tau\bm\theta}_u^{N}||^2+\left(\gamma-(\frac{\alpha^2}{k_m}+\frac{\beta^2}{\theta_m})\right)\tau\sum\limits_{n=1}^{N-1}||{\partial_\tau\bm\theta}_u^{n+1}-{\partial_\tau\bm\theta}_u^{n}||^2\notag\\
\le& C\tau\sum\limits_{n=1}^{N-1}||\theta_p^{n+1} ||^2+C\tau \sum\limits_{n=1}^{N-1}||\theta_T^{n+1}||^2+C\tau\sum\limits_{n=1}^{N-1} ||\bm\theta_u^{n}||_1^2+ C(\tau^2+h^{2k_1+2}+h^{2k_2+2}+h^{2k_3+2})\notag\\
&+(c_0+2b_0)||\theta_p^{1}||^2+(a_0+b_0)||\theta_T^{1}||^2+ C||\bm\epsilon(\bm\theta_u^{1})||^2 +a({\bm\theta}_u^{1},{\bm\theta}_u^{1})+\gamma\tau||{\partial_\tau\bm\theta}_u^{1}||^2\notag\\
&+\frac{\mu}{2}||\bm\epsilon(\bm\theta_u^{3})||^2+\frac{\gamma}{4\tau}||\bm\theta_u^{2}-\bm\theta_u^{1}||^2+\frac{\gamma}{4\tau}||\bm\theta_u^{3}-\bm\theta_u^{2}||^2+\delta\tau\sum\limits_{n=1}^{N-1}||\theta_T^{n+1}||^2.\label{z1}
\end{align}
In \eqref{z1}, choosing $N=2$, we get
\begin{align*}
\frac{\gamma}{4\tau}||\bm\theta_u^{2}-\bm\theta_u^{1}||^2
\le&C(\tau^2+h^{2k_1+2}+h^{2k_2+2}+h^{2k_3+2})\notag\\
&+(c_0+b_0)||\theta_p^{1}||^2+(a_0+b_0)||\theta_T^{1}||^2+ a({\bm\theta}_u^{1},{\bm\theta}_u^{1})+\gamma\tau||{\partial_\tau\bm\theta}_u^{1}||^2.
\end{align*}
In \eqref{z1}, choosing $N=3$, it follows from the Gronwall inequality that
\begin{align*}
\frac{\mu}{4}||\bm\epsilon(\bm\theta_u^{3})||^2+\frac{\gamma}{4\tau}||\bm\theta_u^{3}-\bm\theta_u^{2}||^2
\le&C(\tau^2+h^{2k_1+2}+h^{2k_2+2}+h^{2k_3+2})+(c_0+b_0)||\theta_p^{1}||^2\notag\\
&+(a_0+b_0)||\theta_T^{1}||^2+ a({\bm\theta}_u^{1},{\bm\theta}_u^{1})+\gamma\tau||{\partial_\tau\bm\theta}_u^{1}||^2.
\end{align*}
By the Gronwall inequality, Lemma \ref{lem3.7} and $\bm\theta_u^0= \bm 0$, where $c_0-2b_0>0$, $a_0-2b_0\ge0$ and $\gamma=2(\frac{\alpha^2}{k_m}+\frac{\beta^2}{\theta_m})$, we obtain
\begin{align}
&\mu||\bm\epsilon(\bm\theta_u^{N})||^2+\lambda||\nabla\cdot\bm\theta_u^{N}||^2+||p(t_{n})-p_h^n||^2+||T(t_{n})-T_h^n||^2\notag\\
\le &C(\tau^2+h^{2k_1}+h^{2k_2+2}+h^{2k_3+2})+C||p_h^{1}-R_pp^1||^2+C||T_h^{1}-R_TT^1||^2\notag\\
&+2\mu||\bm\epsilon(\mathbf u_h^{1}-R_u\mathbf u^1)||_1^2+\lambda||\nabla\cdot(\mathbf u_h^{1}-R_u\mathbf u^1)||^2+\gamma\tau||\frac{\mathbf u_h^{1}-R_u\mathbf u^1}{\tau_0}||^2.
\end{align}

Given that the solution to \eqref{b7}-\eqref{b9} at time $t_1$ is $({\mathbf u}_h^1,p_h^1,T_h^1)\in {\mathbf V}_h\times Q_h\times S_h$. The following error estimates are then available \cite{MR4432106},
\begin{align*}
&||\mathbf u_h^{1}-R_u\mathbf u^1||_1^2+||p_h^{1}-R_pp^1||^2+||T_h^{1}-R_TT^1||^2+\tau||\frac{\mathbf u_h^{1}-R_u\mathbf u^1}{\tau}||_1^2\\
\leq &C(\tau^2+h^{2k_1}+h^{2k_2+2}+h^{2k_3+2}).
\end{align*}
Thus, the following error estimate is satisfied by the numerical solution of \eqref{s1}-\eqref{s3},
\begin{align*}
||\bm\epsilon(\mathbf u^n-\mathbf u_h^n)||+||\nabla\cdot(\mathbf u^n-\mathbf u_h^n)||+||p^{n}-p_h^n||+||T^{n}-T_h^n||\le C(\tau+h^{k_1}+h^{k_2+1}+h^{k_3+1}).
\end{align*}
The proof is completed.
\end{proof}

\section{Numerical simulations}\label{NE}
To verify our theoretical findings, we provide numerical tests in this section.  The computational infrastructure and technical assistance required for this work are generously provided by the State Key Laboratory of Mathematical Sciences (SKLMS), Chinese Academy of Sciences, high performance computers.  Two $18$-core $2.3$ GHz Intel Xeon Gold $6140$ CPUs and $192$ GB of RAM are installed on each computer node.
\subsection{Test for convergence order and computational efficiency}\label{eg1}
Consider the model \eqref{b1}-\eqref{b3} on the domain $\Omega=[0,1]^2$ and the last time $t_f=1$. We set $\lambda=1$, $\mu=1$, $\mathbf K=\begin{pmatrix} 1 & 0 \\ 0 & 1 \end{pmatrix}$, $\bm\Theta=\begin{pmatrix} 1 & 0 \\ 0 & 1 \end{pmatrix}$ and $\gamma=2(\alpha^2+\beta^2)$. Here, the values of $\alpha$, $\beta$ and $b_0$ are set as Table \ref{table1} with $a_0=c_0=3b_0$. These constitute five distinct parameter configurations, PA1-PA5, selected based on the reference \cite{2019Monolithic}. The five parameter choices are designed to vary the coupling strength between equations, thereby testing the algorithm across varying coupling intensities. Each right-hand side term is derived from the exact solution
$$
\displaystyle\mathbf u=\binom{\exp(-t)\sin(\pi x)\sin(\pi y)}{\exp(-t)\sin(\pi x)\sin(\pi y)},
\quad
p=t\sin(\pi x)\sin(\pi y),
\quad
T=\exp(-t)\sin(\pi x)\sin(\pi y).
$$
\begin{table}[ht!]
\caption{Parameter choices.}\label{table1}
  \begin{center}
  \begin{tabular}{cccccc}
   \hline
        &          PA1&       PA2&       PA3&          PA4&           PA5 \\ \hline
$\alpha$&          1.0&       0.1&       0.1&          1.0&           0.1 \\
$\beta$ &          1.0&       0.1&       1.0&          0.1&           0.1 \\
$b_0$   &          1.0&       1.0&       0.1&          0.1&           0.1 \\
\hline
  \end{tabular}
  \end{center}
\end{table}
The finite element discretization is built upon a uniform triangular mesh $\mathcal{T}_h$, with polynomial degree set to $k_1=k_2=k_3=1$.

In order to illustrate the convergence orders with respect to time, we fix the spatial mesh size $h=\frac {1} {128}$. Table \ref {table16} demonstrates the first-order temporal convergence in the $L^2$ norm for displacement, pressure and temperature under parameter PA1.

\begin{table}[htpb]
         \centering
	\caption{ Time convergence rates under parameter PA1.}
          \label{table16}
            \begin{center}
  \begin{tabular}{ccccccccc}
   \hline
  $\tau$&   $||{\mathbf u}(t_n)-{\mathbf u}_h^n||$&    R&   $||p(t_n)-p_h^n||$&         R&    $||T(t_n)-T_h^n||$&          R \\   \hline
$\frac{1}{4}$&     6.96538e-02&         -&    9.15642e-03&        -&     2.01854e-02&   -\\
$\frac{1}{8}$&     2.38386e-03&   1.5469 &    3.39852e-03&   1.4299&     1.07458e-02 &   0.9095\\
$\frac{1}{16}$&    1.21089e-03&   0.9772&    1.63428e-03&    1.0563&     5.53216e-03&    0.9579\\
$\frac{1}{32}$&    6.00244e-04&   1.0124&    8.14615e-04&    1.0045&     2.81619e-03&    0.9741\\
$\frac{1}{64}$&    3.10061e-04&   0.9530&    4.24473e-04&    0.9404&     1.43744e-03&    0.9702\\
   \hline
          \end{tabular}
            \end{center}
        \end{table}

Under the temporal discretization parameter $\tau=h^2$, the convergence rates of spatial errors are shown in Tables \ref{table11}-\ref{table15}. These tables demonstrate that our proposed approach achieves the optimal convergence orders within each of the five various coupling strengths, validating the theoretical predictions of Theorem \ref{thm1}.

\begin{table}[ht!]
         \centering
	\caption{Errors and convergence rates for parameter selection PA1.}
          \label{table11}
            \begin{center}
  \begin{tabular}{ccccccccc}
   \hline
  $h$&   $||{\mathbf u}(t_n)-{\mathbf u}_h^n||_1$&    R&   $||p(t_n)-p_h^n||$&         R&    $||T(t_n)-T_h^n||$&          R \\    \hline
$\frac{1}{4}$&             5.10646e-01&             -&         7.13192e-02&               -&               4.70078e-02&   -\\
$\frac{1}{8}$&             2.60347e-01&       0.9719 &         1.86375e-02&               1.9361&          1.34903e-02 &       1.8007\\
$\frac{1}{16}$&            1.30805e-01&             0.9930&    4.71218e-03&               1.9837&          3.49001e-03&                  1.9509\\
$\frac{1}{32}$&           6.54817e-02&             0.9983&    1.18140e-03&               1.9959&          8.79915e-04&                  1.9878\\
\hline
  \end{tabular}
  \end{center}
\end{table}

\begin{table}[ht!]
         \centering
	\caption{Errors and convergence rates for parameter selection PA2.}
          \label{table12}
            \begin{center}
  \begin{tabular}{ccccccccc}
   \hline
  $h$&   $||{\mathbf u}(t_n)-{\mathbf u}_h^n||_1$&    R&   $||p(t_n)-p_h^n||$&         R&    $||T(t_n)-T_h^n||$&          R \\    \hline
$\frac{1}{4}$&             5.10452e-01&             -&         7.36992e-02&               -&               4.64843e-02&   -\\
$\frac{1}{8}$&             2.60322e-01&       0.9715 &  1.94076e-02&               1.9250&          1.33349e-02 &       1.8015\\
$\frac{1}{16}$&            1.30802e-01&      0.9929&    4.91741e-03&               1.9807&          3.44800e-03&        1.9514\\
$\frac{1}{32}$&            6.54813e-02&       0.9982&   1.23353e-03&               1.9951&          8.69234e-04&         1.9879\\
\hline
  \end{tabular}
  \end{center}
\end{table}

\begin{table}[ht!]
\centering
	\caption{Errors and convergence rates for parameter selection PA3.}
          \label{table13}
            \begin{center}
  \begin{tabular}{ccccccccc}
   \hline
$h$&   $||{\mathbf u}(t_n)-{\mathbf u}_h^n||_1$&    R&   $||p(t_n)-p_h^n||$&         R&           $||T(t_n)-T_h^n||$&          R \\    \hline
$\frac{1}{4}$&            5.10619e-01&             -&   7.82814e-02&               -&               4.20207e-02&   -\\
$\frac{1}{8}$&            2.60349e-01&       0.9718 &    2.08770e-02&               1.9068&         1.19126e-02 &       1.8186\\
$\frac{1}{16}$&           1.30806e-01&      0.9930&    5.30921e-03&               1.9753&           3.06888e-03&        1.9567\\
$\frac{1}{32}$&           6.54817e-02&       0.9983&    1.33310e-03&               1.9937&          7.72890e-04&         1.9894\\
\hline
  \end{tabular}
  \end{center}
\end{table}

\begin{table}[ht!]
         \centering
	\caption{Errors and convergence rates for parameter selection PA4.}
          \label{table14}
            \begin{center}
  \begin{tabular}{ccccccccc}
   \hline
$h$&   $||{\mathbf u}(t_n)-{\mathbf u}_h^n||_1$&    R&   $||p(t_n)-p_h^n||$&         R&           $||T(t_n)-T_h^n||$&          R \\    \hline
$\frac{1}{4}$&             5.10853e-01&             -&   7.80993e-02&               -&               4.18617e-02&   -\\
$\frac{1}{8}$&             2.60391e-01&       0.9722 &   2.07978e-02&               1.9089&          1.18389e-02 &       1.8221\\
$\frac{1}{16}$&            1.30811e-01&      0.9932&     5.28653e-03&               1.9760&          3.04735e-03&        1.9579\\
$\frac{1}{32}$&            6.54825e-02&       0.9983&    1.32723e-03&               1.9939&          7.67295e-04&         1.9897\\
\hline
  \end{tabular}
  \end{center}
\end{table}

\begin{table}[ht!]
         \centering
	\caption{Errors and convergence rates for parameter selection PA5.}
          \label{table15}
            \begin{center}
  \begin{tabular}{ccccccccc}
   \hline
 $h$&   $||{\mathbf u}(t_n)-{\mathbf u}_h^n||_1$&    R&   $||p(t_n)-p_h^n||$&         R&           $||T(t_n)-T_h^n||$&          R \\    \hline
$\frac{1}{4}$&             5.10678e-01&             -&   7.82749e-02&               -&               4.19224e-02&   -\\
$\frac{1}{8}$&             2.60352e-01&       0.9720 &   2.08743e-02&               1.9068&          1.18692e-02 &       1.8205\\
$\frac{1}{16}$&            1.30806e-01&      0.9930&     5.30842e-03&               1.9754&          3.05629e-03&        1.9574\\
$\frac{1}{32}$&            6.54818e-02&       0.9983&    1.33289e-03&               1.9937&          7.69621e-04&         1.9896\\
\hline
  \end{tabular}
  \end{center}
\end{table}

To contrast with our approach, we additionally implement a fully implicit nonlinear numerical scheme for the governing system \eqref{b1}-\eqref{b3}, which combines standard finite element method (FEM) for spatial discretization with the backward Euler scheme for temporal integration. Figure \ref{fig3-1} displays the errors of the fully implicit FEM and the sequential finite element technique (SFEM), showing that both approaches have comparable error accuracy and convergence orders at the same grid scale. Furthermore, Table \ref{table17} reports the CPU runtime for the two approaches under parameter PA1, demonstrating that SFEM is computationally more efficient than the fully implicit FEM under the same grid size (and also with the similar level of precision).

\begin{figure}[ht!]
	\centering
	\includegraphics[width=3.5in]{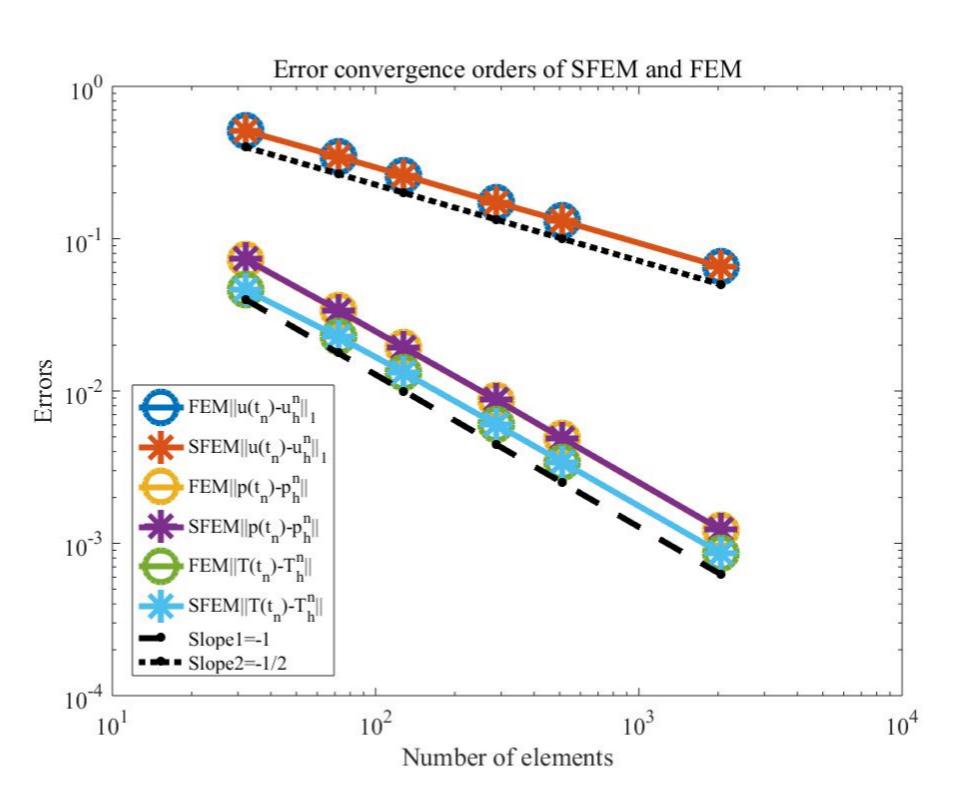}
\caption{\small{Orders of error convergence for SFEM and fully implicit FEM.}}\label{fig3-1}
\end{figure}

\begin{table}[ht]
         \centering
	\caption{Computational efficiency comparison for PA1 parameter selection.}
          \label{table17}
            \begin{center}
  \begin{tabular}{ccccccccc}
   \hline
   \backslashbox{$h$}{CPU time (s)} & SFEM  & FEM  \\
\hline
$\frac{1}{8}$&  2.984 &3.228\\

$\frac{1}{16}$&            33.680      &43.574 \\

$\frac{1}{32}$&           431.919 &716.763 \\
 \hline
  \end{tabular}
  \end{center}
\end{table}
\subsection{Influence of the stabilization parameter $\gamma$}\label{eg2}
Using the same setup as Section \ref{eg1}, we now investigate the sensitivity of the proposed sequential scheme with respect to the stabilization parameter $\gamma$, which is introduced in \eqref{s3} to control the error propagation between the subproblems. According to Theorem \ref{thm0}, the theoretical choice is $\gamma \ge 2(\frac{\alpha^2}{k_m} + \frac{\beta^2}{\theta_m})$.

Firstly, we fix $\alpha=\beta=1$ and vary $\gamma$ from 0.001 to 0.1-spanning three orders of magnitude below the theoretical value. The results in Table \ref{table18} show that the errors and convergence rates are nearly identical across all three cases: the displacement converges at $O(h)$ in the $H^1$-norm, and the pressure and temperature converge at $O(h^2)$ in the $L^2$-norm.

 \begin{table}[ht!]
         \centering
	\caption{Errors and convergence rates for $\alpha=\beta=1$.}
          \label{table18}
            \begin{center}
  \begin{tabular}{ccccccccc}
   \hline
  $h$&   $||{\mathbf u}(t_n)-{\mathbf u}_h^n||_1$&    R&   $||p(t_n)-p_h^n||$&         R&    $||T(t_n)-T_h^n||$&          R \\
   \hline
   $\gamma=0.001$\\    \hline
$\frac{1}{4}$&             5.11341e-01&             -&         7.12231e-02&               -&               4.68136e-02&   -\\
$\frac{1}{8}$&             2.60472e-01&       0.9732 &         1.85926e-02&               1.9376&          1.32483e-02 &       1.8039\\
$\frac{1}{16}$&            1.30822e-01&             0.9935&    4.69920e-03&               1.9842&          3.46561e-03&                  1.9519\\
$\frac{1}{32}$&           6.54838e-02&             0.9984&    1.17804e-03&               1.9960&          8.73614e-04&                  1.9880\\
 \hline
   $\gamma=0.01$\\    \hline
$\frac{1}{4}$&             5.11340e-01&             -&         7.12231e-02&               -&               4.68136e-02&   -\\
$\frac{1}{8}$&             2.60471e-01&       0.9732 &         1.85926e-02&               1.9376&          1.34076e-02 &       1.8039\\
$\frac{1}{16}$&            1.30822e-01&             0.9935&    4.69920e-03&               1.9842&          3.46561e-03&                  1.9519\\
$\frac{1}{32}$&           6.54838e-02&             0.9984&    1.17804e-03&               1.9960&          8.73615e-04&                  1.9880\\
 \hline
   $\gamma=0.1$\\    \hline
$\frac{1}{4}$&             5.16087e-01&             -&         7.05526e-02&               -&               4.53814e-02&   -\\
$\frac{1}{8}$&             2.61373e-01&       0.9815 &         1.82497e-02&               1.9508&          1.26957e-02 &       1.8378\\
$\frac{1}{16}$&            1.30949e-01&             0.9971&    4.59619e-03&               1.9894&          3.25701e-03&                  1.9627\\
$\frac{1}{32}$&           6.55000e-02&             0.9994&    1.15163e-03&               1.9968&          8.20083e-04&                  1.9897\\
\hline
  \end{tabular}
  \end{center}
\end{table}

Secondly, we test stronger coupling regimes with $\alpha=\beta=3, 5, 6$. We choose $\gamma= 36$, $100$, $144$ to examine whether the method remains stable and accurate. Table \ref{table19} demonstrates that the optimal convergence rates are preserved in all cases. As $\alpha$ and $\beta$ increase, the displacement error grows slightly, which is expected due to stronger coupling, but the convergence rates remain optimal. These results indicate that the theoretical bound on the stabilization parameter $\gamma$ is conservative. The method performs robustly for a broad range of
$\gamma \ge 2(\frac{\alpha^2}{k_m} + \frac{\beta^2}{\theta_m})$, including those significantly below the theoretical recommendation. In the practical computation, we adopt $\gamma= 2(\frac{\alpha^2}{k_m} + \frac{\beta^2}{\theta_m})$ as a safe default, while noting that the method tolerates substantial deviations without losing accuracy or stability.
\begin{table}[ht!]
         \centering
	\caption{Errors and convergence rates for $\alpha=\beta=3, 5, 6$.}
          \label{table19}
            \begin{center}
  \begin{tabular}{ccccccccc}
   \hline
  $h$&   $||{\mathbf u}(t_n)-{\mathbf u}_h^n||_1$&    R&   $||p(t_n)-p_h^n||$&         R&    $||T(t_n)-T_h^n||$&          R \\    \hline
   $\gamma=36,\alpha=3,\beta=3$\\    \hline
$\frac{1}{4}$&             5.28209e-01&             -&         7.05736e-02&               -&               4.55127e-02&   -\\
$\frac{1}{8}$&             2.62382e-01&       1.0094 &         1.82320e-02&               1.9527&          1.26363e-02 &       1.8487\\
$\frac{1}{16}$&            1.31006e-01&             1.0020&    4.60291e-03&               1.9859&          3.26883e-03&                  1.9507\\
$\frac{1}{32}$&           6.55085e-02&             0.9999&    1.15338e-03&               1.9967&          8.22991e-04&                  1.9898\\
\hline
 $\gamma=100,\alpha=5,\beta=5$\\    \hline
$\frac{1}{4}$&             7.4402e-01&             -&         6.99593e-02&               -&               4.26958e-02&   -\\
$\frac{1}{8}$&             2.68384e-01&       1.4523 &         1.97586e-02&               1.8240&          1.47933e-02 &       1.5292\\
$\frac{1}{16}$&            1.31929e-01&            1.0245&    4.61008e-03&               2.0996&          3.17824e-03&                2.2186\\
$\frac{1}{32}$&           6.56640e-02&             1.0066&    1.09453e-03&               2.0745&          7.05422e-04&                  2.1717\\
\hline
  $\gamma=144,\alpha=6,\beta=6$\\    \hline
$\frac{1}{4}$&             8.38853e-01&             -&         7.00915e-02&               -&               4.24586e-02&   -\\
$\frac{1}{8}$&             2.86346e-01&       1.5507 &         1.99234e-02&               1.8148&          1.44400e-02 &       1.5560\\
$\frac{1}{16}$&            1.33983e-01&             1.0957&    4.23789e-03&               2.2330&          2.56376e-03&                  2.4937\\
$\frac{1}{32}$&           6.58272e-02&             1.0253&    1.08148e-03&               1.9703&          6.77403e-04&                  1.9202\\
\hline
  \end{tabular}
  \end{center}
\end{table}
\subsection{Robustness with respect to model parameters}\label{eg3}
In this subsection, to further validate the parameter robustness of the proposed sequential scheme, we first test the algorithms under varying settings of hydraulic conductivity $\mathbf K$ and effective thermal conductivity $\mathbf\Theta$. We consider three cases: extremely low permeability ($\mathbf K = 10^{-6}\mathbf I$), extremely low thermal conductivity ($\mathbf\Theta = 10^{-6}\mathbf I$), and the simultaneous extreme case ($\mathbf K = \mathbf\Theta = 10^{-6}\mathbf I$), with all other parameters fixed as PA1. The results are reported in Tables \ref{table20}-\ref{table22}. In all the three cases, the method achieves optimal convergence orders: $O(h)$ for the displacement in the $H^1$-norm and $O(h^2)$ for both pressure and temperature in the $L^2$-norm, identical to the baseline case (Table \ref{table11}). This confirms that the asymptotic convergence behavior is independent of the magnitudes of $\mathbf K$ and $\mathbf\Theta$. These findings align with the parameter-robustness observations in the preconditioning study \cite{MR5009266} and the iterative L-scheme results in \cite{2019Monolithic}, confirming that the proposed sequential method is robust across wide variations in all key model parameters, including $\mathbf K$, $\mathbf\Theta$, and the coupling coefficients.
\begin{table}[ht!]
         \centering
	\caption{Errors and convergence rates for $\mathbf K=10^{-6}\mathbf I$.}
          \label{table20}
            \begin{center}
  \begin{tabular}{ccccccccc}
   \hline
  $h$&   $||{\mathbf u}(t_n)-{\mathbf u}_h^n||_1$&    R&   $||p(t_n)-p_h^n||$&         R&    $||T(t_n)-T_h^n||$&          R \\    \hline
$\frac{1}{4}$&             5.13445e-01&             -&         6.55717e-02&               -&               2.67761e-02&   -\\
$\frac{1}{8}$&             2.60845e-01&       0.9770 &         1.82852e-02&               1.8424&          8.07294e-02 &       1.7298\\
$\frac{1}{16}$&            1.30852e-01&             0.9953&    4.60429e-03&               1.9896&          2.10866e-02&                  1.9368\\
$\frac{1}{32}$&           6.54852e-02&             0.9987&    1.15754e-03&               1.9919&          5.44519e-03&                  1.9533\\
\hline
  \end{tabular}
  \end{center}
\end{table}

\begin{table}[ht!]
         \centering
	\caption{Errors and convergence rates for $\mathbf\Theta=10^{-6}\mathbf I$.}
          \label{table21}
            \begin{center}
  \begin{tabular}{ccccccccc}
   \hline
  $h$&   $||{\mathbf u}(t_n)-{\mathbf u}_h^n||_1$&    R&   $||p(t_n)-p_h^n||$&         R&    $||T(t_n)-T_h^n||$&          R \\    \hline
$\frac{1}{4}$&             5.10451e-01&             -&         4.13480e-02&               -&               3.16718e-02&   -\\
$\frac{1}{8}$&             2.60308e-01&       0.9716 &         9.63858e-03&               2.1009&          8.62271e-03 &       1.8770\\
$\frac{1}{16}$&            1.30800e-01&             0.9929&    2.34308e-03&               2.0404&          2.20581e-03&                  1.9668\\
$\frac{1}{32}$&           6.54809e-02&             0.9982&    5.79084e-04&               2.0166&          5.54693e-04&                  1.9915\\
\hline
  \end{tabular}
  \end{center}
\end{table}

\begin{table}[ht!]
         \centering
	\caption{Errors and convergence rates for $\mathbf K=10^{-6}\mathbf I$ and $\mathbf\Theta=10^{-6}\mathbf I$.}
          \label{table22}
            \begin{center}
  \begin{tabular}{ccccccccc}
   \hline
  $h$&   $||{\mathbf u}(t_n)-{\mathbf u}_h^n||_1$&    R&   $||p(t_n)-p_h^n||$&         R&    $||T(t_n)-T_h^n||$&          R \\    \hline
$\frac{1}{4}$&             5.11189e-01&             -&         9.64956e-02&               -&               6.15402e-02&   -\\
$\frac{1}{8}$&             2.60458e-01&       0.9728 &         2.44771e-02&               1.9790&          1.68016e-02 &       1.8729\\
$\frac{1}{16}$&            1.30821e-01&             0.9935&    6.15793e-03&               1.9909&          4.32621e-03&                  1.9574\\
$\frac{1}{32}$&           6.54838e-02&             0.9984&    1.54811e-03&               1.9919&          1.09495e-03&                  1.9822\\
\hline
  \end{tabular}
  \end{center}
\end{table}

Next, we investigate the influence of the effective thermal parameter $a_0$, the thermal dilation coefficient $b_0$, and the specific storage coefficient $c_0$ on the accuracy of the algorithms. Notably, the method remains stable and accurate even when $a_0 = b_0 = c_0 = 0$ (Table \ref{table23}). This behavior aligns with the observations in \cite{cai} for iterative splitting schemes, and with the parameter-robust preconditioning results in \cite{MR5009266}, which demonstrate the uniform stability of the four-field formulation across wide parameter variations.
\begin{table}[ht!]
         \centering
	\caption{Errors and convergence rates for $a_0=b_0=c_0=0$.}
          \label{table23}
            \begin{center}
  \begin{tabular}{ccccccccc}
   \hline
  $h$&   $||{\mathbf u}(t_n)-{\mathbf u}_h^n||_1$&    R&   $||p(t_n)-p_h^n||$&         R&    $||T(t_n)-T_h^n||$&          R \\    \hline
$\frac{1}{4}$&             5.11213e-01&             -&         7.89788e-02&               -&               4.11325e-02&   -\\
$\frac{1}{8}$&             2.60458e-01&       0.9729 &         2.10853e-02&               1.9052&          1.15835e-02 &       1.8282\\
$\frac{1}{16}$&            1.30820e-01&             0.9935&    5.36372e-03&               1.9749&          2.97758e-03&                  1.9599\\
$\frac{1}{32}$&           6.54836e-02&             0.9984&    1.34688e-03&               1.9936&          7.49455e-04&                  1.9902\\
\hline
  \end{tabular}
  \end{center}
\end{table}

\section*{Acknowledgments}
The authors gratefully acknowledge partial financial support from the National Natural Science Foundation of China (No. 12301465), the Natural Science Foundation of Shandong Province (No. ZR2022MA081, ZR2024MA056, ZR2025LZH001, ZR2025MS1036) and the Research Foundation for Beijing University of Technology New Faculty (No. 006000514122516).

\section*{Declarations}

\subsection*{Conflict of interest}
The authors declare that there are no competing interests associated with this work.

\subsection*{Data availability}
The data produced and examined in this study can be obtained from the corresponding author upon reasonable request.


\end{document}